\documentclass[12pt]{amsart}
\usepackage{graphicx}
\usepackage{amsfonts}
\usepackage{amssymb}
\usepackage{mathrsfs}
\usepackage{enumerate}
\usepackage[cp1250]{inputenc}
\usepackage{amsrefs}
\newtheorem{theorem}{Theorem}[section]
\newtheorem{proposition}[theorem]{Proposition}
\newtheorem{corollary}[theorem]{Corollary}
\newtheorem{lemma}[theorem]{Lemma}
\theoremstyle{definition}

\newtheorem{example}[theorem]{Example}

\numberwithin{equation}{section}
\renewcommand{\gcd}{\mathop{\rm gcd}}
\newcommand{\lcm}{\mathop{\rm lcm}}

\begin{document}
\title{Products of Compressions of $k^{th}$--order slant Toeplitz operators to model spaces}
\author{Bartosz {\L}anucha, Ma{\l}gorzata Michalska}

\address{
Bartosz {\L}anucha,  \newline Institute of Mathematics,
\newline Maria Curie-Sk{\l}odowska University, \newline pl. M.
Curie-Sk{\l}odowskiej 1, \newline 20-031 Lublin, Poland}
\email{bartosz.lanucha@mail.umcs.pl}

\address{
Ma{\l}gorzata Michalska,  \newline Institute of Mathematics,
\newline Maria Curie-Sk{\l}odowska University, \newline pl. M.
Curie-Sk{\l}odowskiej 1, \newline 20-031 Lublin, Poland}
\email{malgorzata.michalska@mail.umcs.pl}

\begin{abstract}
In this paper we investigate intertwining relations for compressions of $k^{th}$--order slant Toeplitz operators to model spaces. We then ask when  a product of two such compressions is a compression itself.
\end{abstract}
\maketitle

\section{Introduction}

Let $L^2=L^2(\mathbb{T},m)$ and $L^{\infty}=L^{\infty}(\mathbb{T},m)$, where $\mathbb{T}=\{z\in\mathbb{C}:|z|=1\}$ is the unit circle and $m$ is the normalized Lebesgue measure on $\mathbb{T}$. Fix a positive integer $k$. A $k^{th}$--order slant Toeplitz operator with symbol $\displaystyle\varphi=\sum_{n=-\infty}^{+\infty}a_nz^n\in L^\infty$ is the operator $U_\varphi:L^2\to L^2$ represented with respect to the standard monomial basis by the doubly infinite matrix $[a_{ki-j}]_{i,j\in\mathbb{Z}}$ ($a_n$ being the $n$--th Fourier coefficient of $\varphi$). Clearly, for $k=1$ such matrix is a doubly infinite Toeplitz matrix (has constant diagonals) and so in that case $U_\varphi=M_\varphi$ is the classical multiplication operator, $M_\varphi f=\varphi f$.

It is known that the $k^{th}$--order slant Toeplitz operator $U_{\varphi}$ can be expressed as
$$U_{\varphi}f=W_kM_{\varphi}f,\quad f\in L^2,$$
where
$$W_k(z^n)=\begin{cases}
z^m& \text{if }\frac{n}{k}=m\in\mathbb{Z},\\
0&\text{if } \frac{n}{k}\not\in\mathbb{Z}.
\end{cases}$$
Moreover, $U_{\varphi}$ can be defined as above for any $\varphi\in L^2$, in which case it is densely defined (its domain contains $L^\infty$). However, $U_{\varphi}$  extends boundedly to $L^2$ if and only if $\varphi\in L^\infty$.

A systematic study of $k^{th}$--order slant Toeplitz operators for $k=2$ (called simply slant Toeplitz operators) was started by M. C. Ho  \cite{Ho} (see also \citelist{\cite{Ho1}\cite{Ho2}\cite{Ho3}}). He also considered compressions of slant Toeplitz operators to the classical Hardy space $H^2$ in the unit disk $\mathbb{D}=\{z\in\mathbb{C}:|z|<1\}$. Such compressions were then investigated by T. Zegeye and S. C. Arora \cite{ZA}. Slant Toeplitz operators have connections with wavelet theory and dynamical systems (see, e.g., \citelist{\cite{GMW}\cite{Ho1}\cite{V}}).

Generalized slant Toeplitz operators, that is, $k^{th}$--order slant Toeplitz operators for $k\geq2$, were introduced and studied by S. C. Arora and R. Batra \citelist{\cite{AB1}\cite{AB}}. Compressions to $H^2$ were also considered there.

Recall that the compression $V_{\varphi}:H^2\to H^2$ of $k^{th}$--order slant Toeplitz operator $U_{\varphi}$ is defined by
$$V_{\varphi}f=PU_{\varphi}f,\quad f\in H^2,$$
where $P$ is the Sz\"{e}go projection. Again, $V_\varphi$ is densely defined for $\varphi\in L^2$ (its domain contains $H^\infty=H^2\cap L^\infty$) and extends boundedly to $H^2$ if and only if $\varphi\in L^\infty$. For $k=1$, $V_\varphi=T_\varphi$ is the classical Toeplitz operator, $T_\varphi f=P(\varphi f)$.

The more recent papers \cite{GDS1,GDS2} deal with slant Toeplitz operators in multivariable setting, while in \citelist{\cite{LL}\cite{LL2}} the authors investigate commutativity of $k^{th}$--order slant Toeplitz operators.

In recent years, compressions of multiplication operators are intensely studied. In particular, compressions to model spaces, that is, subspaces of the form $K_{\alpha}=H^2\ominus \alpha H^2$, where $\alpha$ is a nonconstant inner function: $\alpha\in H^{\infty}$ and $|\alpha|=1$ a.e. on $\mathbb{T}$ (if $\alpha$ is constant, then $K_\alpha=\{0\}$). These subspaces are the non--trivial subspaces of $H^2$ which are invariant for the backward shift operator $S^{*}=T_{\bar z}$. They are also important in view of their connection with topics such as the B. Sz.-Nagy–C. Foias model theory \cite[Chapter VI]{NG}.

Since the model space $K_\alpha$ is a closed subspace of $H^2$, the functional $f\mapsto f^{(n)}(w)$ is bounded on $K_{\alpha}$ for each $n\in \mathbb{N}_0$ and $w\in\mathbb{D}$. Therefore, there exists a kernel function $k_{w,n}^{\alpha}\in K_\alpha$ such that $f^{(n)}(w)=\langle f,k_{w,n}^{\alpha}\rangle$ for all $f\in K_{\alpha}$. It is not difficult to see that $k_{w,n}^{\alpha}=P_{\alpha}k_{w,n}$, where $k_{w,n}(z)=\frac{n!z^n}{(1-\overline{w}z)^{n+1}}\in H^2$ and $P_{\alpha}$ is the orthogonal projection from $L^2$ onto $K_{\alpha}$.
Recall that $P_\alpha=P-M_\alpha PM_{\overline{\alpha}}$ (see \cite[Corollary 14.13]{fm}). In particular,  $k_{w}(z)=k_{w,0}(z)=\frac{1}{1-\overline{w} z}$ and $ k_{w}^{\alpha}=P_{\alpha}k_{w,0}$ is given by
\begin{equation}\label{kern}
k_w^\alpha(z)=\frac{1-\overline{\alpha(w)}\alpha(z)}{1-\overline{w}z},\quad z\in\mathbb{D}.
\end{equation}
Note that $k_w^\alpha$ belongs to $H^\infty$ for each $w\in\mathbb{D}$ and so $K_\alpha^\infty=K_\alpha\cap H^\infty$ is a dense subset of $K_\alpha$.

It is well known that $K_{\alpha}$ is preserved by the antilinear, isometric involution $C_{\alpha}:L^2\to L^2$ (a  conjugation) defined by
\begin{align*}
	C_{\alpha}f(z)=\alpha(z)\overline{z}\overline{f(z)},\quad |z|=1,
\end{align*}
(see \cite{gp} and \cite{bros}*{Chapter 8}). The function $\widetilde{k}_{w,n}^{\alpha}=C_{\alpha}{k}_{w,n}^{\alpha}$ is called the conjugate kernel function. In particular, $\widetilde{k}_{w}^{\alpha}=\widetilde{k}_{w,0}^{\alpha}$ and a.e. on $\mathbb{T}$,
$$ \widetilde{k}_{w}^{\alpha}(z)=C_{\alpha}k_{w}^{\alpha}(z)=\frac{\alpha(z)-\alpha(w)}{z-w}$$
(the above formula is true also for all $z\in\mathbb{D}\setminus\{w\}$ with $\widetilde{k}_{w}^{\alpha}(w)=\alpha'(w)$).

Finally, note that $K_{z^n}$ is the set of all polynomials of degree at most $n-1$. In that case $\dim K_\alpha=n$ and $\{1,z,\ldots,z^{n-1}\}$ is an orthonormal basis for $K_{z^n}$. More details on properties and structure of model spaces can be found in \cite{bros} and \cite{fm}.

For two inner functions $\alpha$, $\beta$ an asymmetric truncated Toeplitz operator $A_{\varphi}^{\alpha,\beta}$ with symbol $\varphi\in L^2$ is defined by
$$A_{\varphi}^{\alpha,\beta}f=P_{\beta}({\varphi}f),\quad f\in K_{\alpha}^{\infty}.$$
Note that $A_{\varphi}^{\alpha,\beta}$ is densely defined since $K_{\alpha}^{\infty}$ is a dense subset of $K_{\alpha}$. Denote
$$\mathcal{T}(\alpha,\beta)=\{A_{\varphi}^{\alpha,\beta}\ \colon\ \varphi\in
L^2\ \mathrm{and}\ A_{\varphi}^{\alpha,\beta}\
\text{extends boundedly to }K_{\alpha}\}$$
and $\mathcal{T}(\alpha)=\mathcal{T}(\alpha,\alpha)$.

Truncated Toeplitz operators, i.e., operators  $A_{\varphi}^{\alpha}=A_{\varphi}^{\alpha,\alpha}$ gained attention in 2007 with D. Sarason's paper \cite{s}. Under intensive study ever since, truncated Toeplitz operators proved to be a rich and interesting topic with many deep results and applications (see \citelist{\cite{gar3}\cite{CFT}} and references therein). Asymmetric truncated Toeplitz operators were introduced later in \citelist{\cite{part}\cite{part2}} and \cite{ptak}. They were then studied for example in \citelist{\cite{ptak2}\cite{BM}\cite{blicharz1}\cite{blicharz2}\cite{BL}\cite{BL3}}.

In this paper we continue the study of compressions of $k^{th}$--order slant Toeplitz operators to model spaces. Fix $k\in\mathbb{N}$. For two inner functions $\alpha$, $\beta$ and for $\varphi\in L^2$ we define
$$U_{\varphi}^{\alpha,\beta}f=P_{\beta}U_{\varphi}f=P_{\beta}W_k(\varphi f),\quad f\in K_{\alpha}^{\infty},$$
and denote
$$\mathcal{S}_k(\alpha,\beta)=\{U_{\varphi}^{\alpha,\beta}\ \colon\ \varphi\in
L^2\ \mathrm{and}\ U_{\varphi}^{\alpha,\beta}\
\text{extends boundedly to }K_{\alpha}\}.$$
Clearly, $\mathcal{S}_1(\alpha,\beta)=\mathcal{T}(\alpha,\beta)$.

The class $\mathcal{S}_k(\alpha,\beta)$ was first introduced in \cite{BM2} where some basic algebraic properties of its elements were investigated. Here we focus on some commuting relations for operators from  $\mathcal{S}_k(\alpha,\beta)$ and on products of this kind of operators. Products of truncated Toeplitz and asymmetric truncated Toeplitz operators were considered in \cite{CT} and \cite{Y}, respectively.

In Section 2 we introduce a shift invariance analogue which can be used to characterize operators from $\mathcal{S}_k(\alpha,\beta)$.

In Section 3 we investigate some intertwining relations involving operators from $\mathcal{S}_k(\alpha,\beta)$ and compressed shifts.

Sections 4--5 are devoted to products of operators from  $\mathcal{S}_k(\alpha,\beta)$ and related problems.

 In what follows we will say that an inner function $\beta$ divides $\alpha$ if $\alpha/{\beta}=\alpha\overline{\beta}$ is also an inner function. In that case we will write $\beta\leq\alpha$. Moreover, for arbitrary inner functions $\alpha,\beta$ we will denote by $\gcd(\alpha,\beta)$ and $\lcm(\alpha,\beta)$ their greatest common divisor and their least common multiple, respectively. See \cite{berc} for more details on the arithmetic of inner functions.

\section{Shift invariance for operators from $\mathcal{S}_k(\alpha,\beta)$}

Recall that a bounded linear operator $A:K_\alpha\to K_\beta$ is shift invariant if
$$\langle ASf,Sg\rangle=\langle Af,g\rangle$$
for all $f\in K_\alpha$, $g\in K_\beta$ such that $Sf\in K_\alpha$ and $Sg\in K_\beta$. For $\alpha=\beta$ this notion was considered by D. Sarason in \cite{s}. He proved that a bounded linear operator $A:K_\alpha\to K_\alpha$ is a truncated Toeplitz operator if and only if it is shift invariant. For the asymmetric case the notion of shift invariance was considered in \cite{ptak} ($\beta\leq\alpha$), \cite{BL} ($\alpha,\beta$ - finite Blaschke products) and \cite{BM}. There it was proved that the operators from $\mathcal{T}(\alpha,\beta)$ are also characterized by shift invariance. Here we prove the following generalization of this result.

\begin{theorem}
	\label{thm_char_U}
	Let $\alpha, \beta$ be two nonconstant inner functions and let $U:K_\alpha\to K_\beta$ be a bounded linear operator. Then $U\in\mathcal{S}_k(\alpha,\beta)$, $k\in\mathbb{N}$, if and only if
	\begin{equation}\label{eq_char_U}
	\langle US^kf,Sg\rangle=\langle Uf,g\rangle
	\end{equation}
	for all $f\in K_\alpha$, $g\in K_\beta$ such that $S^kf\in K_\alpha$ and $Sg\in K_\beta$.
\end{theorem}
As  in the case of truncated Toeplitz operators, the proof is based on a characterization of operators from $\mathcal{S}_k(\alpha,\beta)$ in terms
of operators $S_\alpha$ and $S_\beta$. It was proved in \cite{BM2} that a bounded linear operator $U:K_\alpha\to K_\beta$ belongs to $\mathcal{S}_k(\alpha,\beta)$ if and only if there exist $\chi\in K_\alpha$ and $\psi_0,\ldots,\psi_{k-1}\in K_{\beta}$ such that
\begin{equation}\label{warr}
U-S_\beta^*US_\alpha^k=\widetilde{k}_{0}^\beta\otimes\chi+\sum_{j=0}^{k-1}\psi_j\otimes\widetilde{k}_{0,j}^\alpha
\end{equation}
(\cite[Cor. 6]{BM2}).

We also need the following lemma.

\begin{lemma}
\label{lem_shift_pow_m}
Let $\alpha$ be a nonconstant inner function and let $m\in\mathbb{N}$. For each $f\in K_\alpha$ we have
\begin{enumerate}
  \item[(a)] $\displaystyle S_{\alpha}^m f(z)=z^mf(z)-\sum_{j=0}^{m-1}\frac{1}{j!} \langle f,\widetilde{k}_{0,j}^\alpha\rangle \alpha z^{m-1-j},\ |z|=1$;
  \item[(b)] $\displaystyle (S_{\alpha}^*)^m f(z)=\overline{z}^mf(z)-\sum_{j=0}^{m-1}\frac{1}{j!} \langle f,k_{0,j}^\alpha\rangle \overline{z}^{m-j},\ |z|=1$.
\end{enumerate}
\end{lemma}
\begin{proof}
Let $f\in K_\alpha$. Since $S_{\alpha}^m =A_{z^m}^{\alpha}$ and $P_\alpha=P-M_{\alpha}PM_{\overline{\alpha}}$, we get
  $$S_{\alpha}^m f=A^{\alpha}_{z^m} f=P_{\alpha}(z^m f)=z^m f-\alpha P(\overline{\alpha}z^m f).$$
Now
\begin{align*}
  P(\overline{\alpha}z^m f)&=P(z^{m-1}\overline{\alpha}z f)
  =P(z^{m-1}\overline{C_{\alpha}f})
  =P\left(\sum_{j=0}^\infty \overline{\langle C_{\alpha}f,z^j\rangle}z^{m-1-j}\right)\\
  &=\sum_{j=0}^{m-1} \frac{1}{j!}\overline{\langle C_{\alpha}f,j!z^j\rangle}z^{m-1-j}
  =\sum_{j=0}^{m-1} \frac{1}{j!}\langle f,C_{\alpha}P_{\alpha}(j!z^j)\rangle z^{m-1-j}\\
  &=\sum_{j=0}^{m-1} \frac{1}{j!}\langle f,\widetilde{k}_{0,j}^\alpha\rangle z^{m-1-j}
\end{align*}
and (a) follows. Part (b) follows from Lemma 2(a) in \cite{BM2} and the fact that for $f\in K_{\alpha}$ we have $(S_{\alpha}^*)^m f=(S^*)^m f$.
\end{proof}

Recall that for $f\in K_{\alpha}$ we have $Sf=zf\in K_\alpha$ if and only if $f\perp \widetilde{k}_0^{\alpha}$. On the other hand, for $f\in K_{\alpha}$ we always have $S^*f\in K_\alpha$ and $\bar z f\in K_\alpha$ if and only if $f\perp {k}_0^{\alpha}$.
\begin{corollary}
\label{cor_shift_pow_m}
Let $\alpha$ be a nonconstant inner function, let $m\in\mathbb{N}$ and let $f\in K_\alpha$. Then
\begin{enumerate}
  \item[(a)] $\displaystyle z^mf(z)\in K_\alpha$ if and only if
        $$f\perp  {\rm span}\{\widetilde{k}_{0,j}^\alpha:\ j=0,1,\ldots,m-1\};$$
  \item[(b)] $\displaystyle \overline{z}^mf(z)\in K_\alpha$ if and only if
        $$f\perp {\rm span}\{k_{0,j}^\alpha:\ j=0,1,\ldots,m-1\}.$$
\end{enumerate}
\end{corollary}

\begin{proof}[Proof of Theorem \ref{thm_char_U}
	]
Let $U\in\mathcal{S}_k(\alpha,\beta)$ and take $f\in K_\alpha$, $g\in K_\beta$ such that $S^kf\in K_\alpha$, $Sg\in K_\beta$. By \cite[Cor. 6]{BM2} there are $\chi\in K_\alpha$ and $\psi_0,\ldots,\psi_{k-1}\in K_{\beta}$ such that  \eqref{warr}  is satisfied.
It follows from Corollary \ref{cor_shift_pow_m}
that $f\perp  {\rm span}\{\widetilde{k}_{0,j}^\alpha:\ j=0,1,\ldots,k-1\}$ and $g\ \perp\ \widetilde{k}_{0}^\beta$, and therefore
\begin{align*}
  \langle Uf,g\rangle-\langle US^kf,Sg\rangle&=\langle (U-S_\beta^*US_\alpha^k)f,g\rangle
  =\left\langle \left(\widetilde{k}_{0}^\beta\otimes\chi+\sum_{j=0}^{k-1}\psi_j\otimes\widetilde{k}_{0,j}^\alpha\right)f,g\right\rangle\\
  &=\langle f,\chi\rangle \cdot\langle \widetilde{k}_{0}^\beta,g\rangle
  +\sum_{j=0}^{k-1}\langle f,\widetilde{k}_{0,j}^\alpha\rangle\cdot\langle\psi_j,g\rangle=0.
\end{align*}
Assume now that $U$ satisfies \eqref{eq_char_U}. As above, for each $f\in K_\alpha$, $g\in K_\beta$ we have
  $$\langle Uf,g\rangle-\langle US^kf,Sg\rangle=\langle (U-S_\beta^*US_\alpha^k)f,g\rangle.$$
Hence \eqref{eq_char_U} means that the operator $T=U-S_\beta^*US_\alpha^k$ maps $f\perp \mathcal{M}={\rm span}\{\widetilde{k}_{0,j}^\alpha:\ j=0,1,\ldots,k-1\}$ to a function $Tf\in\mathcal{N}=\mathbb{C}\cdot \widetilde{k}_{0}^\beta$. Thus
\begin{equation}
\label{eq_char_U_aux1}
  P_{\mathcal{N}^\perp}TP_{\mathcal{M}^\perp}=0,
\end{equation}
where $P_{\mathcal{N}^\perp}$ and $P_{\mathcal{M}^\perp}$ are orthogonal projections onto $\mathcal{N}^\perp=K_\beta\ominus \mathcal{N}$ and $\mathcal{M}^\perp=K_\alpha\ominus \mathcal{M}$, respectively. Clearly,
  $$P_{\mathcal{N}^\perp}=I_{K_\beta}-c\cdot(\widetilde{k}_{0}^\beta\otimes \widetilde{k}_{0}^\beta) \quad\text{where } c=\|\widetilde{k}_{0}^\beta\|^{-1}.$$
It can be verified that there exist complex numbers $a_{i,j}$, $0\leq i,j\leq k-1$, such that
  $$P_{\mathcal{M}^\perp}=I_{K_\alpha}-\sum_{i,j=0}^{k-1} a_{i,j}\cdot (\widetilde{k}_{0,i}^\alpha\otimes \widetilde{k}_{0,j}^\alpha).$$
Note that not only $\{\widetilde{k}_{0,j}^\alpha:j=0,\ldots,k-1\}$ are not orthogonal but they might not even be linearly independent, so it might happen that $a_{i,j}=0$ for some $i$ and $j$. Still, \eqref{eq_char_U_aux1} can be written as
\begin{align*}
  0&=\big(I_{K_\beta}-c\cdot(\widetilde{k}_{0}^\beta\otimes \widetilde{k}_{0}^\beta)\big) T \left(I_{K_\alpha}-\sum_{i,j=0}^{k-1} a_{i,j}\cdot (\widetilde{k}_{0,i}^\alpha\otimes \widetilde{k}_{0,j}^\alpha)\right)\\
  &=\big(I_{K_\beta}-c\cdot(\widetilde{k}_{0}^\beta\otimes \widetilde{k}_{0}^\beta)\big) \left(T-\sum_{i,j=0}^{k-1} a_{i,j}\cdot (T\widetilde{k}_{0,i}^\alpha\otimes \widetilde{k}_{0,j}^\alpha)\right)\\
  &=T-\sum_{i,j=0}^{k-1} a_{i,j}\cdot(T\widetilde{k}_{0,i}^\alpha\otimes \widetilde{k}_{0,j}^\alpha
)  -c\cdot(\widetilde{k}_{0}^\beta\otimes T^*\widetilde{k}_{0}^\beta)
  +c\sum_{i,j=0}^{k-1} a_{i,j}\cdot\langle T\widetilde{k}_{0,i}^\alpha,\widetilde{k}_{0}^\beta\rangle\cdot (\widetilde{k}_{0}^\beta\otimes \widetilde{k}_{0,j}^\alpha).
\end{align*}
Hence
\begin{align*}
  U-S_\beta^*US_\alpha^k=T
  =\widetilde{k}_{0}^\beta\otimes (\bar cT^*\widetilde{k}_{0}^\beta)
  +\sum_{j=0}^{k-1}\left(\sum_{i=0}^{k-1} a_{i,j}\cdot
  \big(T\widetilde{k}_{0,i}^\alpha-c\langle T\widetilde{k}_{0,i}^\alpha,\widetilde{k}_{0}^\beta\rangle\cdot \widetilde{k}_{0}^\beta\big)\right)
  \otimes \widetilde{k}_{0,j}^\alpha
\end{align*}
and $U$ satisfies \eqref{warr} with
$$\chi=\bar cT^*\widetilde{k}_{0}^\beta\quad\text{and}\quad\psi_j=\sum_{i=0}^{k-1} a_{i,j}\cdot
\big(T\widetilde{k}_{0,i}^\alpha-c\langle T\widetilde{k}_{0,i}^\alpha,\widetilde{k}_{0}^\beta\rangle\cdot \widetilde{k}_{0}^\beta\big).$$
Thus by \cite[Cor. 6]{BM2}, $U\in\mathcal{S}_k(\alpha,\beta)$.
\end{proof}

Note that if $\dim K_\alpha\leq k$, then the set $\{\widetilde{k}_{0,j}^\alpha:\ j=0,1,\ldots,k-1\}$ spans $K_\alpha$ and so $f\perp  {\rm span}\{\widetilde{k}_{0,j}^\alpha,\ j=0,1,\ldots,k-1\}$ if and only if $f=0$.

\begin{corollary}
\label{cor_char_U}
Let $\alpha, \beta$ be two nonconstant inner functions and assume that $\dim K_\alpha=m< +\infty$. If $k\geq m$, then every bounded linear operator from $K_\alpha$ into $K_\beta$ belongs to $\mathcal{S}_k(\alpha,\beta)$.
\end{corollary}

\section{Intertwining properties for operators from $\mathcal{S}_k(\alpha,\beta)$}

Let $\alpha$ and $\beta$ be two nonconstant inner functions. By D. Sarasons commutant lifting theorem each bounded linear operator $A:K_\alpha\to K_\alpha$ commuting with the compressed shift $S_\alpha$ is of the form $A=A^{\alpha}_\varphi$ with $\varphi\in H^\infty$. More generally, a bounded linear operator $A:K_\alpha\to K_\beta$ satisfies
\begin{equation}
\label{eqx}
S_\beta A=AS_\alpha\end{equation}
if and only if $A=A^{\alpha,\beta}_\varphi$  with $\varphi\in H^\infty$ such that $\beta\leq\alpha\varphi$ (see \cite[Theorem III.1.16]{berc}). Recently, the authors in \cite{intert} proved using basic methods that $A:K_\alpha\to K_\beta$ satisfies \eqref{eqx} if and only if $A=A^{\alpha,\beta}_\varphi$ with $\varphi\in\tfrac{\beta}{\gcd(\alpha, \beta)}\, K_{\gcd(\alpha, \beta)}$.

Here our goal is to describe (for any fixed positive integer $k$) all bounded linear operators $U:K_\alpha\to K_\beta$ that satisfy
\begin{equation}\label{eq_inv_rel_U}
S_\beta U=US_\alpha^k.
\end{equation}
We use a reasoning similar to the one used in \cite{intert}. First recall that  $U:K_\alpha\to K_\beta$ belongs to $\mathcal{S}_k(\alpha,\beta)$ if and only if
$$S_\beta U- US_\alpha^k=k_0^\beta\otimes \chi +\sum_{j=0}^{k-1} \psi_j\otimes \widetilde{k}_{0,j}^\alpha
$$
for some $\chi\in K_\alpha$ and $\psi_0,\ldots,\psi_{k-1}\in K_{\beta}$ \cite[Corollary 8(b)]{BM2}. Thus each $U$ satisfying \eqref{eq_inv_rel_U} clearly belongs to $\mathcal{S}_k(\alpha,\beta)$. We therefore describe the operators from $\mathcal{S}_k(\alpha,\beta)$ satisfying \eqref{eq_inv_rel_U}.

\begin{proposition}
\label{prop_U_inv_with_shift_and_backshift}
Let $\alpha$ and $\beta$ be two nonconstant inner functions and let $\varphi\in L^2$. Then
\begin{enumerate}
  \item[(a)]  $\displaystyle S_\beta U_\varphi^{\alpha,\beta}-U_\varphi^{\alpha,\beta} S^k_\alpha=\sum_{j=0}^{k-1} \tfrac{1}{j!}P_{\beta}W_k(\alpha\varphi z^{k-1-j})\otimes \widetilde{k}^\alpha_{0,j}-k_0^\beta\otimes P_{\alpha}(\overline{z}^k \overline{\varphi})$,
  \item[(b)]  $\displaystyle S_\beta^* U_\varphi^{\alpha,\beta}-U_\varphi^{\alpha,\beta} (S^*_\alpha)^k=\sum_{j=0}^{k-1} \tfrac{1}{j!}P_{\beta}W_k(\varphi \overline{z}^{k-j})\otimes k^\alpha_{0,j}-\widetilde{k}_0^\beta\otimes P_{\alpha}(\overline{\varphi}\cdot W_k^*\beta)$,
\end{enumerate}
where both equalities hold on $K_{\alpha}^\infty$.
\end{proposition}
\begin{proof}
Let $f\in K^\infty_{\alpha}$ and $g\in K^\infty_{\beta}$. Then
\begin{align*}
  \left\langle S_\beta U_\varphi^{\alpha,\beta} f,g \right\rangle
  &=\left\langle P_\beta W_k (\varphi f),S^*_{\beta}g \right\rangle
  =\left\langle W_k (\varphi f),S^*_{\beta}g \right\rangle
  =\left\langle \varphi f,W^*_k S^*g \right\rangle\\
  &=\left\langle \varphi f,W^*_k(\overline{z}g-\overline{z}g(0)) \right\rangle
  =\left\langle \varphi f,\overline{z}^kW^*_k g\rangle-\langle \varphi f,\overline{z}^kg(0) \right\rangle\\
  &=\left\langle \varphi z^k f, W^*_k g\right\rangle
  -\left\langle f,\overline{\varphi}\overline{z}^k \right\rangle\overline{g(0)}
  =\left\langle \varphi z^k f, W^*_k g\right\rangle
  -\left\langle f,P_{\alpha}(\overline{\varphi}\overline{z}^k) \right\rangle \cdot\langle k_0^\beta,g\rangle\\
  &=\left\langle \varphi z^k f,W^*_kg\right\rangle
  -\left\langle\big(k_0^\beta\otimes P_{\alpha}(\overline{\varphi}\overline{z}^k)\big)f,g\right\rangle.
\end{align*}
Moreover, using Lemma \ref{lem_shift_pow_m}(a), we get
\begin{align*}
  \left\langle U_\varphi^{\alpha,\beta} S^k_\alpha f,g \right\rangle
  &=\left\langle \varphi S^k_\alpha f,W_k^*g \right\rangle
  =\left\langle \varphi z^k f,W_k^*g \right\rangle
  -\sum_{j=0}^{k-1}\tfrac{1}{j!}\langle f,\widetilde{k}_{0,j}^\alpha\rangle\cdot\langle \varphi\alpha z^{k-1-j},W_k^*g\rangle\\
  &=\left\langle \varphi z^k f,W_k^*g \right\rangle
  -\sum_{j=0}^{k-1}\tfrac{1}{j!}\langle f,\widetilde{k}_{0,j}^\alpha\rangle\cdot\langle P_{\beta}W_k(\varphi\alpha z^{k-1-j}),g\rangle\\
  &=\left\langle \varphi z^k f,W_k^*g \right\rangle
  -\left\langle \left(\sum_{j=0}^{k-1}\tfrac{1}{j!}P_{\beta}W_k(\varphi\alpha z^{k-1-j})\otimes\widetilde{k}_{0,j}^\alpha\right)f,g\right\rangle.
\end{align*}
This completes the proof of (a).

To prove (b) note that, by Lemma \ref{lem_shift_pow_m}(a), $S_\beta g=zg-\langle g,\widetilde{k}^\beta_{0}\rangle\beta$. Hence, for $f\in K^\infty_{\alpha}$ and $g\in K^\infty_{\beta}$, we have
\begin{align*}
  \left\langle S^*_\beta U_{\varphi}^{\alpha,\beta} f,g \right\rangle
  &=\left\langle W_k(\varphi f),S_\beta g\right\rangle
  =\left\langle W_k(\varphi f),zg\right\rangle
  -\left\langle W_k(\varphi f),\beta\right\rangle \cdot\overline{\langle g,\widetilde{k}_0^\beta\rangle}\\
  &=\left\langle \varphi f,W_k^*(zg)\right\rangle
  -\left\langle \varphi f,W_k^*\beta\right\rangle \cdot\langle \widetilde{k}_0^\beta,g\rangle\\
  &=\left\langle \varphi\overline{z}^k f,W_k^*g\right\rangle
  -\left\langle\left\langle f,P_\alpha(\overline{\varphi}\cdot W_k^*\beta)\right\rangle\widetilde{k}_0^\beta,g\right\rangle\\
  &=\left\langle \varphi\overline{z}^k f,W_k^*g\right\rangle
  -\langle\big(\widetilde{k}_0^\beta\otimes P_\alpha(\overline{\varphi}\cdot W_k^*\beta)\big)f,g\rangle.
\end{align*}
Moreover, by Lemma \ref{lem_shift_pow_m}(b),
\begin{align*}
  \left\langle U_{\varphi}^{\alpha,\beta} (S^*_\alpha)^k f,g \right\rangle
  &=\left\langle \varphi(S^*_\alpha)^k f,W^*_kg \right\rangle
  =\left\langle \varphi\overline{z}^k f,W^*_kg \right\rangle
  -\sum_{j=0}^{k-1}\tfrac{1}{j!}\left\langle f,k_{0,j}^\alpha \right\rangle\cdot\left\langle \varphi\overline{z}^{k-j},W^*_kg \right\rangle\\
  &=\left\langle \varphi\overline{z}^k f,W^*_kg \right\rangle
  -\sum_{j=0}^{k-1}\tfrac{1}{j!}\left\langle f,k_{0,j}^\alpha \right\rangle\cdot\left\langle P_{\beta}W_k(\varphi\overline{z}^{k-j}),g\right\rangle\\
  &=\left\langle \varphi\overline{z}^k f,W^*_kg \right\rangle
  -\left\langle \left(\sum_{j=0}^{k-1}\tfrac{1}{j!} P_{\beta}W_k(\varphi\overline{z}^{k-j})\otimes k_{0,j}^\alpha\right)f,g\right\rangle
\end{align*}
and (b) follows.
\end{proof}

\begin{corollary}
\label{cor_U_inv_with_compr_shift}
Let $\alpha$ and $\beta$ be two nonconstant inner functions and let $\varphi\in H^2$.
\begin{enumerate}
  \item[(a)] If $W^*_k\beta\leq\alpha$ and $U_{\varphi}^{\alpha,\beta}\in\mathcal{S}_k(\alpha,\beta)$, then
  \begin{equation}
  \label{eq_U_inv_shift}
    S_\beta U_{\varphi}^{\alpha,\beta}=U_{\varphi}^{\alpha,\beta}S^k_\alpha.
  \end{equation}
  \item[(b)] If $\alpha\leq W^*_k\beta$ and $U_{\overline{\varphi}}^{\alpha,\beta}\in\mathcal{S}_k(\alpha,\beta)$, then
  \begin{equation}
  \label{eq_U_inv_back_shift}
    S^*_\beta U_{\overline{\varphi}}^{\alpha,\beta}=U_{\overline{\varphi}}^{\alpha,\beta}(S^*_\alpha)^k.
  \end{equation}
\end{enumerate}
\end{corollary}

\begin{proof}
Let $\varphi\in H^2$.

 $(a)$ Since $\overline{z}^k\overline{\varphi}\in\overline{zH^2}$, we have $P_\alpha(\overline{z}^k\overline{\varphi})=0$. If $W^*_k\beta\leq\alpha$, then by \cite[Lemma 2.1(g)]{BM2} for each $0\leq j\leq k-1$ we get
        $$P_\beta W_k(\alpha\varphi z^{k-1-j})=W_k P_{W_k^*\beta}(\alpha\varphi z^{k-1-j})=0,$$
      since here $\alpha\varphi z^{k-1-j}\in (W_k^*\beta)\cdot H^2$. Thus \eqref{eq_U_inv_shift} holds by Proposition \ref{prop_U_inv_with_shift_and_backshift}(a).

  $(b)$ Here $\overline{\varphi}\overline{z}^{k-j}\in\overline{zH^2}$ for each $0\leq j\leq k-1$ and so
        $$P_{\beta}W_k(\overline{\varphi}\overline{z}^{k-j})=P_{\beta}PW_k(\overline{\varphi}\overline{z}^{k-j})
        =P_{\beta}W_kP(\overline{\varphi}\overline{z}^{k-j})=0.$$
      Moreover, if $\alpha\leq W_k^*\beta$, then $W_k^*\beta\cdot\varphi\in\alpha H^2$ and so $P_{\alpha}(W_k^*\beta\cdot\varphi)=0$. Hence \eqref{eq_U_inv_back_shift} follows from Proposition \ref{prop_U_inv_with_shift_and_backshift}(b) (with $\overline{\varphi}$ in place of $\varphi$).
\end{proof}
Note that for $k=1$ the above corollary is Proposition 3.3 from \cite{ptak}.

In what follows we assume that $\dim K_\alpha\geq k$ (if $\dim K_\alpha\leq k$, then $\mathcal{S}_k(\alpha,\beta)$ contains all bounded linear operators from $K_\alpha$ into $K_\beta$).

Let $U=U_{\varphi}^{\alpha,\beta}\in \mathcal{S}_k(\alpha,\beta)$ with $\varphi\in L^2$. By Proposition \ref{prop_U_inv_with_shift_and_backshift}, $U$ satisfies \eqref{eq_inv_rel_U} if and only if
\begin{equation}
\label{eq_aux_1_inv_rel_U}
  \sum_{j=0}^{k-1}\tfrac{1}{j!} P_\beta W_k(\alpha\varphi z^{k-1-j})\otimes \widetilde{k}^{\alpha}_{0,j}=k_0^\beta\otimes P_\alpha(\overline{z}^k\overline{\varphi}).
\end{equation}
Since $\dim K_\alpha\geq k$, \eqref{eq_aux_1_inv_rel_U} holds if and only if there exist numbers $c_0,c_1,\ldots,c_{k-1}\in\mathbb{C}$ such that
\begin{equation}
\label{eq_aux_2_inv_rel_U}
  \tfrac{1}{j!} P_\beta W_k(\alpha\varphi z^{k-1-j})=c_j\cdot k_0^\beta \quad\text{for each } j\in\{0,1,\ldots,k-1\}
\end{equation}
and
\begin{equation}
\label{eq_aux_3_inv_rel_U}
  P_\alpha(\overline{z}^k\overline{\varphi})=\sum_{j=0}^{k-1}\overline{c}_j \widetilde{k}^{\alpha}_{0,j}.
\end{equation}
Note that \eqref{eq_aux_3_inv_rel_U} happens if and only if
  $$P_{\alpha}(\alpha\varphi z^{k-1})=P_{\alpha}C_{\alpha}(\overline{z}^k\overline{\varphi})
  =C_{\alpha}P_{\alpha}(\overline{z}^k\overline{\varphi})=C_{\alpha}\left(\sum_{j=0}^{k-1}\overline{c}_j \widetilde{k}^{\alpha}_{0,j}\right)
  =\sum_{j=0}^{k-1}c_j k^{\alpha}_{0,j},$$
that is, if and only if
  $$P_{\alpha}\left(\alpha\varphi z^{k-1}-\sum_{j=0}^{k-1}j!c_j z^j\right)=0.$$
This condition can be expressed as
\begin{equation}
\label{eq_aux_4_inv_rel_U}
  \alpha\varphi z^{k-1}-\sum_{j=0}^{k-1}j!c_j z^j\quad \perp\quad K_\alpha.
\end{equation}
Let us now consider \eqref{eq_aux_2_inv_rel_U}, which can be equivalently expressed as
  $$P_\beta\left(W_k(\alpha\varphi z^{k-1-j})-j!c_j\right)=0  \quad\text{for each } j\in\{0,1,\ldots,k-1\},$$
i.e.,
\begin{equation}
\label{eq_aux_5_inv_rel_U}
  W_k(\alpha\varphi z^{k-1-j})-j!c_j\ \perp\ K_\beta  \quad\text{for each } j\in\{0,1,\ldots,k-1\}.
\end{equation}
For each $0\leq j\leq k-1$ denote
  $$\varphi_j=z^j W^*_k W_k(\alpha\varphi z^{k-1}\overline{z}^j).$$
Observe that
  $$\varphi_j=M_{z^j}W^*_k\big(W_k(\alpha\varphi z^{k-1-j})-j!c_j\big)+j!c_jz^j.$$
Since $M_{z^j}$ and $W_k^*$ are isometries, \eqref{eq_aux_5_inv_rel_U} is equivalent to
\begin{equation}
\label{eq_aux_6_inv_rel_U}
  \psi_j:=\varphi_j-j!c_jz^j\ \perp\ z^jW^*_k K_\beta  \quad\text{for each } j\in\{0,1,\ldots,k-1\}.
\end{equation}
Recall that $W_k^*W_k$ is the orthogonal projection from $L^2$ onto the closed linear span of $\{z^{km}: m\in\mathbb{Z}\}$. Hence the functions $\psi_0,\psi_1,\ldots,\psi_{k-1}$ are pairwise orthogonal. The same is true for the subspaces $W_k^*K_\beta,zW_k^*K_\beta,\ldots,z^{k-1}W_k^*K_\beta$. Moreover,
  $$W_k^*K_\beta\oplus zW_k^*K_\beta\oplus\ldots\oplus z^{k-1}W_k^*K_\beta=K_{W^*_k\beta}$$
(see \cite{BM2} for details). It follows that \eqref{eq_aux_6_inv_rel_U} is equivalent to
\begin{equation}
\label{eq_aux_7_inv_rel_U}
  \sum_{j=0}^{k-1}\psi_j=\sum_{j=0}^{k-1}\varphi_j- \sum_{j=0}^{k-1}j!c_jz^j\quad \perp\quad  K_{W^*_k\beta}.
\end{equation}
It is not difficult to verify that for each $f\in L^2$
  $$f=\sum_{j=0}^{k-1} z^jW^*_kW_k(\overline{z}^j f).$$
In particular,
  $$\sum_{j=0}^{k-1}\varphi_j=\sum_{j=0}^{k-1}z^jW^*_kW_k(\overline{z}^j \alpha\varphi z^{k-1})=\alpha\varphi z^{k-1},$$
and \eqref{eq_aux_7_inv_rel_U} can be expressed as
\begin{equation}
\label{eq_aux_8_inv_rel_U}
  \alpha\varphi z^{k-1}-\sum_{j=0}^{k-1}j!c_jz^j\quad \perp\quad  K_{W^*_k\beta}.
\end{equation}
Summing up, $U=U_{\varphi}^{\alpha,\beta}$ satisfies \eqref{eq_inv_rel_U} if and only if there exist $c_0,c_1,\ldots,c_{k-1}$ such that \eqref{eq_aux_4_inv_rel_U} and \eqref{eq_aux_8_inv_rel_U} hold. Equivalently,
  $$\alpha\varphi z^{k-1}-\sum_{j=0}^{k-1}j!c_jz^j\quad \perp\quad  \text{span}\{K_\alpha,K_{W^*_k\beta}\}=K_{\lcm(\alpha, W^*_k\beta)},$$
that is,
  $$\alpha\varphi z^{k-1}-\sum_{j=0}^{k-1}j!c_jz^j\in \overline{zH^2}+\lcm(\alpha, W^*_k\beta)H^2.$$
In other words, $U_{\varphi}^{\alpha,\beta}$ satisfies \eqref{eq_inv_rel_U} if and only if
\begin{align*}
  \varphi\in \ \overline{\alpha}\overline{z}^{k-1}K_{z^k} &+\overline{\alpha{z}^{k}H^2}+\overline{\alpha}\overline{z}^{k-1}\lcm(\alpha, W^*_k\beta)H^2\\
  =\overline{\alpha}\overline{K_{z^k}}
  &+\overline{\alpha{z}^{k}H^2} +\overline{z}^{k-1}\tfrac{W^*_k\beta}{\gcd(\alpha, W^*_k\beta)}H^2\\
  =\ \overline{\alpha H^2}&+\overline{z}^{k-1}\tfrac{W^*_k\beta}{\gcd(\alpha, W^*_k\beta)} \left(\gcd(\alpha, W^*_k\beta)H^2+
  K_{\gcd(\alpha, W^*_k\beta)}\right)\\
  =\ \overline{\alpha H^2}&+\overline{z}^{k-1} (W^*_k\beta) H^2+
  \overline{z}^{k-1}\tfrac{W^*_k\beta}{\gcd(\alpha, W^*_k\beta)}K_{\gcd(\alpha, W^*_k\beta)}
\end{align*}
($K_{z^k}=\text{span}\{1,z,\ldots,z^{k-1}\}$). We have thus proved
\begin{theorem}
\label{thm_inv_rel_U_aux}
Let $\alpha,\beta$ be two nonconstant inner functions and let $k\leq\dim K_\alpha$. Operator $U=U_{\varphi}^{\alpha,\beta}\in\mathcal{S}_k(\alpha,\beta)$ satisfies
\begin{equation*}
  S_\beta U_{\varphi}^{\alpha,\beta}=U_{\varphi}^{\alpha,\beta}S_\alpha^k
\end{equation*}
if and only if
\begin{equation*}
  \varphi\in\overline{\alpha H^2}+\overline{z}^{k-1} (W^*_k\beta)H^2+
  \overline{z}^{k-1}\tfrac{W^*_k\beta}{\gcd(\alpha, W^*_k\beta)}K_{\gcd(\alpha, W^*_k\beta)}.
\end{equation*}
\end{theorem}

As $U_{\varphi}^{\alpha,\beta}=0$ for $\varphi\in \overline{\alpha H^2}+\overline{z}^{k-1} (W^*_k\beta) H^2$ we get
\begin{theorem}
\label{thm_inv_rel_U}
Let $\alpha,\beta$ be two nonconstant inner functions, let $k\leq\dim K_\alpha$ and let $U:K_\alpha\to K_\beta$ be a bounded linear operator. Then
  $$S_\beta U=US_\alpha^k$$
if and only if $U\in\mathcal{S}_k(\alpha,\beta)$ and $U=U_{\varphi}^{\alpha,\beta}$ with
\begin{equation*}
  \varphi\in\overline{z}^{k-1}\tfrac{W^*_k\beta}{\gcd(\alpha, W^*_k\beta)}K_{\gcd(\alpha, W^*_k\beta)}.
\end{equation*}
\end{theorem}
Note that for $k=1$ we obtain Corollary 3.7 and Theorem 3.9 from \cite{ptak}. Recall that in that case $S_\beta A=AS_\alpha$ if and only if $A\in\mathcal{T}(\alpha,\beta)$ and $A=A_{\varphi}^{\alpha,\beta}$ with $\varphi\in\tfrac{\beta}{\gcd(\alpha, \beta)}K_{\gcd(\alpha, \beta)}$. In particular, each operator intertwining $S_\beta$ and $S_\alpha$ is an asymmetric truncated Toeplitz operator with an analytic symbol. Note that for $k>1$ analicity of the symbol is not guaranteed.
\begin{example}
For any $k>1$ let $a\in\mathbb{D}\setminus\{0\}$. Put $\alpha(z)=z^{2k}$ and $\beta(z)=z^2$. Then
  $$W^*_k\beta=z^{2k}=\alpha=\gcd(\alpha, W^*_k\beta)$$
and
  $$\overline{\alpha H^2}+\overline{z}^{k-1} (W^*_k\beta) H^2+
  \overline{z}^{k-1}\tfrac{W^*_k\beta}{\gcd(\alpha, W^*_k\beta)}K_{\gcd(\alpha, W^*_k\beta)}
  =\overline{z^{2k} H^2}+z^{k+1}H^2+
  \overline{z}^{k-1}K_{z^{2k}}.$$
Hence $S_\beta U_{\varphi}^{\alpha,\beta}=U_{\varphi}^{\alpha,\beta}S_\alpha^k$ if for example $\varphi(z)=\overline{z}^{k-1}$.
\end{example}

\begin{corollary}
	Let $\alpha,\beta$ be two nonconstant inner functions, $k\leq\dim K_\beta$ and let $U:K_\alpha\to K_\beta$ be a bounded linear operator. Then
	$$(S_\beta^*)^k U=US_\alpha^*$$
	if and only if $U^*\in\mathcal{S}_k(\beta,\alpha)$ and $U=(U_{\varphi}^{\beta,\alpha})^*$ with
	\begin{equation*}
	\varphi\in\overline{z}^{k-1}\tfrac{W^*_k\alpha}{\gcd(\beta, W^*_k\alpha)}K_{\gcd(\beta, W^*_k\alpha)}.
	\end{equation*}
\end{corollary}

\begin{theorem}
\label{thm_inv_rel_U_*_aux}
Let $\alpha,\beta$ be two nonconstant inner functions and let $k\leq\dim K_\alpha$. Operator $U=U_{\varphi}^{\alpha,\beta}\in\mathcal{S}_k(\alpha,\beta)$ satisfies
\begin{equation}
\label{eq_inv_rel_U_*}
  S^*_\beta U_{\varphi}^{\alpha,\beta}=U_{\varphi}^{\alpha,\beta}(S^*_\alpha)^k
\end{equation}
if and only if
\begin{equation*}
  \varphi\in\overline{\alpha H^2}+\overline{z}^{k-1} (W^*_k\beta) H^2+
  \overline{\left(\tfrac{\alpha}{\gcd(\alpha, W^*_k\beta)}K_{\gcd(\alpha, W^*_k\beta)}\right)}.
\end{equation*}
\end{theorem}

\begin{proof}
Let $U=U_{\varphi}^{\alpha,\beta}\in \mathcal{S}_k(\alpha,\beta)$, where $\varphi\in L^2$. By Proposition \ref{prop_U_inv_with_shift_and_backshift}, $U$ satisfies \eqref{eq_inv_rel_U_*} if and only if
\begin{equation}
\label{eq_aux_1_inv_rel_U_*}
  \sum_{j=0}^{k-1}\tfrac{1}{j!} P_\beta W_k(\varphi \overline{z}^{k-j})\otimes k^{\alpha}_{0,j}=\widetilde{k}_0^\beta\otimes P_\alpha(\overline{\varphi}\cdot W_k^*\beta).
\end{equation}
Since $\dim K_\alpha\geq k$, \eqref{eq_aux_1_inv_rel_U_*} holds if and only if there exist numbers $c_0,c_1,\ldots,c_{k-1}\in\mathbb{C}$ such that
\begin{equation}
\label{eq_aux_2_inv_rel_U_*}
  \tfrac{1}{j!} P_\beta W_k(\varphi \overline{z}^{k-j})=c_j\cdot \widetilde{k}_0^\beta \quad\text{for each } j\in\{0,1,\ldots,k-1\}
  \end{equation}
  and
  \begin{equation}
\label{eq_aux_3_inv_rel_U_*}
  P_\alpha(\overline{\varphi}\cdot W_k^*\beta)=\sum_{j=0}^{k-1}\overline{c}_j k^{\alpha}_{0,j}.
\end{equation}
Equality \eqref{eq_aux_3_inv_rel_U_*} happens if and only if
\begin{equation}
\label{eq_aux_4_inv_rel_U_*}
  \overline{\varphi}\cdot W_k^*\beta-\sum_{j=0}^{k-1}j!\overline{c}_j z^j\ \perp\ K_\alpha.
\end{equation}
Moreover, \eqref{eq_aux_2_inv_rel_U_*} holds if and only if
\begin{equation*}
  W_k(\varphi \overline{z}^{k-j})-j!c_j\beta\overline{z}\ \perp\ K_\beta  \quad\text{for each } j\in\{0,1,\ldots,k-1\},
\end{equation*}
which is equivalent to
\begin{equation*}
  W^*_kW_k(\varphi \overline{z}^{k-j})-j!c_j(W^*_k\beta)\overline{z}^k\ \perp\ W^*_kK_\beta  \quad\text{for each } j\in\{0,1,\ldots,k-1\},
\end{equation*}
or
\begin{equation*}
  z^{k-j}W^*_kW_k(\varphi \overline{z}^{k-j})-j!c_j(W^*_k\beta)\overline{z}^j\ \perp\ z^{k-j}W^*_kK_\beta  \quad\text{for each } j\in\{0,1,\ldots,k-1\}.
\end{equation*}
It follows that
\begin{equation*}
  \varphi- (W^*_k\beta)\sum_{j=0}^{k-1}j!c_j\overline{z}^j\ \perp\  zK_{W^*_k\beta}
\end{equation*}
and
\begin{equation*}
  \overline{(W^*_k\beta)}\varphi -\sum_{j=0}^{k-1}j!c_j\overline{z}^j\ \perp\  \overline{(W^*_k\beta)}z K_{W^*_k\beta}=\overline{K_{W^*_k\beta}},
\end{equation*}
which can also be expressed as
\begin{equation}
\label{eq_aux_8_inv_rel_U_*}
  (W^*_k\beta)\overline{\varphi} -\sum_{j=0}^{k-1}j!\overline{c}_jz^j\ \perp\  K_{W^*_k\beta}.
\end{equation}
Summing up, $U=U_{\varphi}^{\alpha,\beta}$ satisfies \eqref{eq_inv_rel_U_*} if and only if there exist $c_0,c_1,\ldots,c_{k-1}$ such that \eqref{eq_aux_4_inv_rel_U_*} and \eqref{eq_aux_8_inv_rel_U_*} hold, that is,
  $$(W^*_k\beta)\overline{\varphi}-\sum_{j=0}^{k-1}j!\overline{c}_jz^j\ \perp\  \textrm{span}\{K_\alpha, K_{W^*_k\beta}\}=K_{\textrm{lcm}(\alpha, W^*_k\beta)}.$$
Equivalently,
  $$(W^*_k\beta)\overline{\varphi} \in K_{z^k}+ \overline{zH^2}+\textrm{lcm}(\alpha, W^*_k\beta)H^2.$$
Thus, $U_{\varphi}^{\alpha,\beta}$ satisfies \eqref{eq_inv_rel_U_*} if and only if
\begin{align*}
  \overline{\varphi}\in \ &\overline{(W^*_k\beta)}K_{z^k} +\overline{(W^*_k\beta) z H^2}+\overline{W^*_k\beta}\cdot\textrm{lcm}(\alpha, W^*_k\beta)H^2\\
  &=z^{k-1}\overline{(W^*_k\beta)K_{z^k}} +z^{k-1}\overline{(W^*_k\beta) z^k H^2}+\tfrac{\alpha}{\textrm{gcd}(\alpha, W^*_k\beta)}H^2\\
  &=z^{k-1}\overline{(W^*_k\beta) H^2}+\tfrac{\alpha}{\textrm{gcd}(\alpha, W^*_k\beta)}\left(\textrm{gcd}(\alpha, W^*_k\beta)H^2+
  K_{\textrm{gcd}(\alpha, W^*_k\beta)}\right)\\
  &=z^{k-1}\overline{(W^*_k\beta) H^2}+\alpha H^2+
  \tfrac{\alpha}{\textrm{gcd}(\alpha, W^*_k\beta)}K_{\textrm{gcd}(\alpha, W^*_k\beta)}
\end{align*}
and so
  $$\varphi\in \overline{\alpha H^2}+\overline{z}^{k-1} (W^*_k\beta) H^2+
  \overline{\left(\tfrac{\alpha}{\textrm{gcd}(\alpha, W^*_k\beta)}K_{\textrm{gcd}(\alpha, W^*_k\beta)}\right)},$$
which completes the proof.
\end{proof}

Since $U_{\varphi}^{\alpha,\beta}=0$ for $\varphi\in \overline{\alpha H^2}+\overline{z}^{k-1} (W^*_k\beta) H^2$  and
  $$\overline{\left(\tfrac{\alpha}{\textrm{gcd}(\alpha, W^*_k\beta)}K_{\textrm{gcd}(\alpha, W^*_k\beta)}\right)}
  =\overline{\alpha} \cdot\textrm{gcd}(\alpha, W^*_k\beta) \overline{K_{\textrm{gcd}(\alpha, W^*_k\beta)}}
  =\overline{\alpha} zK_{\textrm{gcd}(\alpha, W^*_k\beta)}$$
we get

\begin{theorem}
\label{thm_inv_rel_U_*}
Let $\alpha,\beta$ be two nonconstant inner functions, let $k\leq\dim K_\alpha$ and let $U:K_\alpha\to K_\beta$ be a bounded linear operator. Then
  $$S^*_\beta U=U(S^*_\alpha)^k$$
if and only if $U\in\mathcal{S}_k(\alpha,\beta)$ and $U=U_{\varphi}^{\alpha,\beta}$ with
\begin{equation*}
  \varphi\in\overline{\alpha} zK_{\textrm{gcd}(\alpha, W^*_k\beta)}.
\end{equation*}
\end{theorem}

\begin{corollary}
	Let $\alpha,\beta$ be two nonconstant inner functions, let $k\leq\dim K_\beta$ and let $U:K_\alpha\to K_\beta$ be a bounded linear operator. Then
	$$S_\beta^k U=US_\alpha$$
	if and only if $U^*\in\mathcal{S}_k(\beta,\alpha)$ and $U=(U_{\varphi}^{\beta,\alpha})^*$ with
	\begin{equation*}
	\varphi\in\overline{\beta} zK_{\textrm{gcd}(\beta, W^*_k\alpha)}.
	\end{equation*}
\end{corollary}

\section{Products of operators from $\mathcal{S}_k(\alpha,\beta)$ with analytic or anti-analytic symbols}
In this Section we consider products of operators with analytic or anti-analytic symbols, which belong to $\mathcal{S}_k(\alpha,\beta)$. We start with some auxiliary results.

Recall that for a nonconstant inner function $\alpha$ we have (see \cite{s}):
\begin{equation}
\label{eq_def_op}
I_{K_{\alpha}}-S_{\alpha}S^*_{\alpha}={k}_0^\alpha\otimes {k}_{0}^\alpha\end{equation}
and\begin{equation}
\label{eq_def_con_op}
I_{K_{\alpha}}-S^*_{\alpha}S_{\alpha}=\widetilde{k}_0^\alpha\otimes \widetilde{k}_{0}^\alpha.
\end{equation}

In what follows we will use the following lemma.
\begin{lemma}
\label{lem_sar_k}
Let $\alpha$ be a nonconstant inner function and let $m\in\mathbb{N}$. Then
\begin{align}
\label{eq_sar_k}
	I_{K_{\alpha}}-S_{\alpha}^m(S^*_{\alpha})^m
	=\sum_{j=0}^{m-1}P_\alpha (z^j)\otimes P_\alpha (z^j)
	=\sum_{j=0}^{m-1}\left(\tfrac{1}{j!}\right)^2\!\!\cdot(k_{0,j}^\alpha\otimes k_{0,j}^\alpha),\\
\label{eq_sar_k_conj}
	I_{K_{\alpha}}-(S^*_{\alpha})^mS_{\alpha}^m
	=\sum_{j=0}^{m-1}C_\alpha P_\alpha (z^j)\otimes C_\alpha P_\alpha (z^j)
    =\sum_{j=0}^{m-1}\left(\tfrac{1}{j!}\right)^2\!\!\cdot( \widetilde{k}_{0,j}^\alpha\otimes \widetilde{k}_{0,j}^\alpha).
	\end{align}
\end{lemma}
\begin{proof}
To prove \eqref{eq_sar_k} observe that from \eqref{eq_def_op} we have
	$$S_{\alpha}^j(S^*_{\alpha})^j-S_{\alpha}^{j+1}(S^*_{\alpha})^{j+1}=\left(S_{\alpha}^j{k}_0^\alpha\right)\otimes \left(S_{\alpha}^j{k}_{0}^\alpha\right)$$
for any nonnegative integer $j$. Adding the above equalities for $j=0,1,\ldots,m-1$ we get
\begin{align*}
	I_{K_{\alpha}}-S_{\alpha}^m(S^*_{\alpha})^m
	&=\sum_{j=0}^{m-1}(S_{\alpha}^{j} k_0^\alpha)\otimes (S_{\alpha}^{j} k_0^\alpha),
\end{align*}
where $S_{\alpha}^0{k}_0^\alpha={k}_0^\alpha$. Moreover, for any positive integer $j$,
\begin{equation}
\label{eq_sar_k_aux}
	S_{\alpha}^{j} k_0^\alpha=A^{\alpha}_{z^j} k_0^\alpha=P_\alpha (z^j-\overline{\alpha(0)}\alpha z^j)
    =P_\alpha (z^j)=\tfrac{1}{j!}\cdot P_\alpha (j!z^j)
	=\tfrac{1}{j!}\cdot k_{0,j}^\alpha
\end{equation}
and thus \eqref{eq_sar_k} holds.
	To prove \eqref{eq_sar_k_conj} we apply $C_{\alpha}$ to \eqref{eq_sar_k} and use $C_\alpha$-symmetry of $S_\alpha$ ($C_\alpha S_\alpha C_\alpha=S_\alpha^*$, see \cite{gp}).
\end{proof}

Let $k$ and $m$ be two arbitrary fixed positive integers.
\begin{proposition}
	\label{thm_product_S_km}
	Let $\alpha$, $\beta$ and $\gamma$ be nonconstant inner functions and let $\varphi,\psi\in H^2$.
	\begin{enumerate}
		\item[(a)] Assume that $U_{\varphi}^{\beta,\gamma}\in\mathcal{S}_m(\beta,\gamma)$ and $U_{\psi}^{\alpha,\beta}\in\mathcal{S}_k(\alpha,\beta)$. If $W_m^*\gamma\leq\beta$ and $W_k^*\beta\leq\alpha$, then $U=U_{\varphi}^{\beta,\gamma}U_{\psi}^{\alpha,\beta}\in\mathcal{S}_{km}(\alpha,\gamma)$ and $U=U_{\eta}^{\alpha,\gamma}$ with
		$$\eta=\sum_{j=0}^{km-1}\frac1{j!}(W_{km}^*U k_{0,j}^\alpha) \overline{z}^j.$$
		
		\item[(b)] Assume that $U_{\overline{\varphi}}^{\beta,\gamma}\in\mathcal{S}_m(\beta,\gamma)$ and $U_{\overline{\psi}}^{\alpha,\beta}\in\mathcal{S}_k(\alpha,\beta)$. If $\alpha\leq W_k^*\beta$ and $\beta\leq W_m^*\gamma$, then $U=U_{\overline{\varphi}}^{\beta,\gamma}U_{\overline{\psi}}^{\alpha,\beta}\in\mathcal{S}_{km}(\alpha,\gamma)$ and $U=U_{\zeta}^{\alpha,\gamma}$ with
		$$\zeta=\overline{\alpha}\sum_{j=0}^{km-1}\frac{1}{j!} (W_{km}^*U\widetilde{k}_{0,j}^\alpha)z^{j+1}.$$
	\end{enumerate}
\end{proposition}

Before giving a proof of Proposition \ref{thm_product_S_km} note that if a bounded linear operator $U:K_\alpha\to K_\beta$ satisfies \eqref{warr} with some $\chi\in K_\alpha$ and $\psi_0,\ldots,\psi_{k-1}\in K_\beta$, then $U=U_{\varphi}^{\alpha,\beta}$ with
\begin{equation}\label{2ox}
\varphi=(W_k^*\beta)\overline{z}^k\overline{\chi} +\overline{\alpha}\sum_{j=0}^{k-1} (W_k^*\psi_{j})j!z^{j+1}
\end{equation}
(\cite[Theorem 1]{BM2}).

Similarly, operators from $\mathcal{S}_k(\alpha,\beta)$ can be characterized as follows (\cite[Corollary 7]{BM2}): $U\in\mathcal{S}_k(\alpha,\beta)$ if and only if there exist functions $\chi\in K_\alpha$ and $\psi_0,\ldots,\psi_{k-1}\in K_\beta$ such that
\begin{equation}\label{warr2}
 U-S_\beta U(S_\alpha^*)^k=k_0^\beta\otimes \chi +\sum_{j=0}^{k-1} \psi_j\otimes k_{0,j}^\alpha.
 \end{equation}
In that case, $U=U_{\varphi}^{\alpha,\beta}$ with

\begin{equation}\label{ox}
\varphi=\overline{\chi} +\sum_{j=0}^{k-1} (W_k^*\psi_{j})j!\overline{z}^j
\end{equation}
(\cite[Corollary 4]{BM2}).

\begin{proof}[Proof of Proposition \ref{thm_product_S_km}]
(a) Let $U=U_{\varphi}^{\beta,\gamma}U_{\psi}^{\alpha,\beta}$. By Corollary \ref{cor_U_inv_with_compr_shift}(a) and formula \eqref{eq_sar_k} (Lemma \ref{lem_sar_k}) we have
\begin{align*}
U -S_{\gamma}U(S^*_{\alpha})^{km}&=
  U_{\varphi}^{\beta,\gamma}U_{\psi}^{\alpha,\beta} -S_{\gamma}U_{\varphi}^{\beta,\gamma}U_{\psi}^{\alpha,\beta}(S^*_{\alpha})^{km}=U_{\varphi}^{\beta,\gamma}U_{\psi}^{\alpha,\beta} -U_{\varphi}^{\beta,\gamma}S^m_{\beta}U_{\psi}^{\alpha,\beta}(S^*_{\alpha})^{km}\\
  &=U_{\varphi}^{\beta,\gamma}U_{\psi}^{\alpha,\beta} -U_{\varphi}^{\beta,\gamma}U_{\psi}^{\alpha,\beta}S_{\alpha}^{km}(S^*_{\alpha})^{km}
  =U(I_{K_\alpha}-S_{\alpha}^{km}(S^*_{\alpha})^{km})\\
  &=U\left(\sum_{j=0}^{km-1}\left(\frac{1}{j!}\right)^2 (k_{0,j}^\alpha\otimes k_{0,j}^\alpha)\right)=\sum_{j=0}^{km-1}\left(\frac{1}{j!}\right)^2 (Uk_{0,j}^\alpha\otimes k_{0,j}^\alpha).
\end{align*}
Therefore  $U=U_{\varphi}^{\beta,\gamma}U_{\psi}^{\alpha,\beta}= U_{\eta}^{\alpha,\gamma}\in\mathcal{S}_{km}(\alpha,\gamma)$ (see \eqref{warr2}) where, by \eqref{ox},
\begin{equation*}
\eta=\sum_{j=0}^{km-1}\frac{1}{j!}  (W_{km}^*Uk_{0,j}^\alpha)\overline{z}^j.
\end{equation*}

(b) Let $U=U_{\overline{\varphi}}^{\beta,\gamma}U_{\overline{\psi}}^{\alpha,\beta}$. By Corollary \ref{cor_U_inv_with_compr_shift}(b) and formula \eqref{eq_sar_k_conj},
\begin{align*}
U -S^*_{\gamma}US_{\alpha}^{km}&=U_{\overline{\varphi}}^{\beta,\gamma}U_{\overline{\psi}}^{\alpha,\beta} -S^*_{\gamma}U_{\overline{\varphi}}^{\beta,\gamma} U_{\overline{\psi}}^{\alpha,\beta}S_{\alpha}^{km}=U_{\overline{\varphi}}^{\beta,\gamma}U_{\overline{\psi}}^{\alpha,\beta} -U_{\overline{\varphi}}^{\beta,\gamma} U_{\overline{\psi}}^{\alpha,\beta}(S^*_{\alpha})^{km} S_{\alpha}^{km}\\
&  =U (I-(S^*_{\alpha})^{km} S_{\alpha}^{km})=\sum_{j=0}^{km-1}\left(\frac{1}{j!}\right)^2(U\widetilde{k}_{0,j}^\alpha\otimes \widetilde{k}_{0,j}^\alpha).
\end{align*}
Thus $U_{\overline{\varphi}}^{\beta,\gamma}U_{\overline{\psi}}^{\alpha,\beta}= U_{\zeta}^{\alpha,\gamma}\in\mathcal{S}_{km}(\alpha,\gamma)$ (see \eqref{warr}) where, by \eqref{2ox},
\begin{equation*}
\zeta=\overline{\alpha}\sum_{j=0}^{km-1}\frac{1}{j!} (W_{km}^*U\widetilde{k}_{0,j}^\alpha)z^{j+1}.
\end{equation*}
\end{proof}

\begin{corollary}
	Let $\alpha$, $\beta$ and $\gamma$ be nonconstant inner functions and let $\varphi,\psi\in H^2$.
	\begin{enumerate}
		\item[(a)] Assume that $U_{\varphi}^{\beta,\gamma}\in\mathcal{S}_m(\beta,\gamma)$ and $U_{\psi}^{\alpha,\beta}\in\mathcal{S}_k(\alpha,\beta)$. If $W_m^*\gamma\leq\beta$ and $W_k^*\beta\leq\alpha$, then $U=U_{\varphi}^{\beta,\gamma}U_{\psi}^{\alpha,\beta}\in\mathcal{S}_{km}(\alpha,\gamma)$ and $U=U_{\eta}^{\alpha,\gamma}$ with
		$$\eta=\sum_{j=0}^{km-1}W_{km}^*W_{km}P_{W_{km}^*\gamma}\big((W_{k}^*\varphi)P_{W_{k}^*\beta}(\psi z^j)\big)\cdot \overline{z}^j.$$
		\item[(b)] Assume that $U_{\overline{\varphi}}^{\beta,\gamma}\in\mathcal{S}_m(\beta,\gamma)$ and $U_{\overline{\psi}}^{\alpha,\beta}\in\mathcal{S}_k(\alpha,\beta)$. If $\alpha\leq W_k^*\beta$ and $\beta\leq W_m^*\gamma$, then $U=U_{\overline{\varphi}}^{\beta,\gamma}U_{\overline{\psi}}^{\alpha,\beta}\in\mathcal{S}_{km}(\alpha,\gamma)$ and $U=U_{\zeta}^{\alpha,\gamma}$ with
		 $$\zeta=\overline{\alpha}\sum_{j=0}^{km-1}W_{km}^*W_{km}P_{W_{km}^*\gamma}\big(\overline{(W_{k}^*\varphi)}\overline{\psi}P_{\alpha}(\alpha\overline{z}^{j+1})\big)\cdot z^{j+1}.$$
	\end{enumerate}
\end{corollary}

\begin{proof}
	(a) Let $U=U_{\varphi}^{\beta,\gamma}U_{\psi}^{\alpha,\beta}$, where $\varphi,\psi\in H^2$,  $U_{\varphi}^{\beta,\gamma}\in\mathcal{S}_m(\beta,\gamma)$ and $U_{\psi}^{\alpha,\beta}\in\mathcal{S}_k(\alpha,\beta)$. Then for each $f\in K_{\gamma}^\infty$,
	
	\begin{align*}
	\langle \tfrac1{j!}U k_{0,j}^\alpha ,f\rangle&=\langle U_{\varphi}^{\beta,\gamma}U_{\psi}^{\alpha,\beta} P_\alpha(z^j) ,f\rangle=\langle U_{\psi}^{\alpha,\beta} P_\alpha(z^j) ,(U_{\varphi}^{\beta,\gamma})^*f\rangle\\
	&=\langle P_{\beta}W_k(\psi P_\alpha(z^j)) ,P_{\beta}(\overline{\varphi}W_m^*f)\rangle=\langle \psi P_\alpha(z^j) ,W_k^*P_{\beta}(\overline{\varphi}W_m^*f)\rangle\\
	&=\langle  P_\alpha(z^j) ,\overline{\psi}P_{W_k^*\beta}W_k^*(\overline{\varphi}W_m^*f)\rangle=\langle z^j , P_\alpha\big(\overline{\psi}P_{W_k^*\beta}(\overline{(W_k^*\varphi)}W_k^*W_m^*f)\big)\rangle.
	\end{align*}
	Recall that $W_k^*\beta\leq\alpha$, which implies that $P_{W_k^*\beta}(\overline{(W_k^*\varphi)}W_k^*W_m^*f)\in K_\alpha$. Since $\psi\in H^2$, we see that $P\big(\overline{\psi}P_{W_k^*\beta}(\overline{(W_k^*\varphi)}W_k^*W_m^*f)\big)$ is orthogonal to $\alpha H^2$ and so	
	 $$P_\alpha\big(\overline{\psi}P_{W_k^*\beta}(\overline{(W_k^*\varphi)}W_k^*W_m^*f)\big)=P\big(\overline{\psi}P_{W_k^*\beta}(\overline{(W_k^*\varphi)}W_k^*W_m^*f)\big).$$
	Hence, using the fact that $W_k^*W_m^*=W_{km}^*$, we obtain
	
	\begin{align*}
	\langle \tfrac1{j!}U k_{0,j}^\alpha ,f\rangle&=\langle z^j , P\big(\overline{\psi}P_{W_k^*\beta}(\overline{(W_k^*\varphi)}W_{km}^*f)\big)\rangle=\langle\psi z^j , P_{W_k^*\beta}(\overline{(W_k^*\varphi)}W_{km}^*f)\rangle\\
	&=\langle (W_k^*\varphi)P_{W_k^*\beta}(\psi z^j) , W_{km}^*f\rangle=\langle P_{\gamma}W_{km}\big((W_k^*\varphi)P_{W_k^*\beta}(\psi z^j)\big) , f\rangle.
	\end{align*}
	Thus
	\begin{align*}
	\frac{1}{j!}  W_{km}^*Uk_{0,j}^\alpha&=W_{km}^*P_{\gamma}W_{km}\big((W_k^*\varphi)P_{W_k^*\beta}(\psi z^j)\big)\\
	&=W_{km}^*W_{km}P_{W_{km}^*\gamma}\big((W_k^*\varphi)P_{W_k^*\beta}(\psi z^j)\big).
	\end{align*}
	By Proposition \ref{thm_product_S_km}(a), $U= U_{\eta}^{\alpha,\gamma}\in\mathcal{S}_{km}(\alpha,\gamma)$ with
	\begin{equation*}
	\eta=\sum_{j=0}^{km-1}W_{km}^*W_{km}P_{W_{km}^*\gamma}\big((W_m^*\varphi)P_{W_k^*\beta}(\psi z^j)\big)\cdot\overline{z}^j.
	\end{equation*}
	
	(b) Let $U=U_{\overline{\varphi}}^{\beta,\gamma}U_{\overline{\psi}}^{\alpha,\beta}$, where $\varphi,\psi\in H^2$,  $U_{\overline{\varphi}}^{\beta,\gamma}\in\mathcal{S}_m(\beta,\gamma)$ and $U_{\overline{\psi}}^{\alpha,\beta}\in\mathcal{S}_k(\alpha,\beta)$. Then for each $f\in K_{\gamma}^\infty$,
	
	\begin{align*}
	\langle \tfrac1{j!}U \widetilde{k}_{0,j}^\alpha ,f\rangle&=\langle U_{\overline{\varphi}}^{\beta,\gamma}U_{\overline{\psi}}^{\alpha,\beta} C_\alpha P_\alpha(z^j) ,f\rangle=\langle U_{\overline{\psi}}^{\alpha,\beta} C_\alpha P_\alpha(z^j) ,(U_{\overline{\varphi}}^{\beta,\gamma})^*f\rangle\\
	&=\langle P_{\beta}W_k(\overline{\psi}C_\alpha P_\alpha(z^j)) ,P_{\beta}({\varphi}W_m^*f)\rangle=\langle \overline{\psi} C_\alpha P_\alpha(z^j) ,W_k^*P_{\beta}(\varphi W_m^*f)\rangle\\
&=\langle \overline{\psi} C_\alpha P_\alpha(z^j),P_{W_k^*\beta}W_k^*({\varphi}W_m^*f)\rangle=\langle P_{W_k^*\beta}\big(\overline{\psi} C_\alpha P_\alpha(z^j)\big),(W_k^*{\varphi})W_k^*W_m^*f\rangle.
 	\end{align*}
	Here $\alpha\leq W_k^*\beta$, which implies that $$P_{W_k^*\beta}\big(\overline{\psi} C_\alpha P_\alpha(z^j)\big)=P\big(\overline{\psi} C_\alpha P_\alpha(z^j)\big),$$
	and so
	\begin{align*}
	\langle \tfrac1{j!}U \widetilde{k}_{0,j}^\alpha ,f\rangle&=\langle P\big(\overline{\psi} C_\alpha P_\alpha(z^j)\big),(W_k^*{\varphi})W_k^*W_m^*f\rangle=\langle \overline{(W_k^*{\varphi})}\overline{\psi} C_\alpha P_\alpha(z^j),W_{km}^*f\rangle\\
	&=\langle P_{\gamma}W_{km}\big( \overline{(W_k^*{\varphi})}\overline{\psi}P_\alpha C_\alpha (z^j)\big),f\rangle.
	\end{align*}
	Thus
	\begin{align*}
	\frac{1}{j!}  W_{km}^*U\widetilde{k}_{0,j}^\alpha&=W_{km}^*P_{\gamma}W_{km}\big(\overline{(W_k^*{\varphi})}\overline{\psi}P_\alpha C_\alpha (z^j)\big)\\
	&=W_{km}^*W_{km}P_{W_{km}^*\gamma}\big(\overline{(W_k^*{\varphi})}\overline{\psi}P_\alpha (\alpha \overline{z}^{j+1})\big).
	\end{align*}
	By Proposition \ref{thm_product_S_km}(b), $U= U_{\zeta}^{\alpha,\gamma}\in\mathcal{S}_{km}(\alpha,\gamma)$ with
	\begin{equation*}
	\zeta=\overline{\alpha}\sum_{j=0}^{km-1}W_{km}^*W_{km}P_{W_{km}^*\gamma}\big(\overline{(W_k^*{\varphi})}\overline{\psi}P_\alpha (\alpha \overline{z}^{j+1})\big)\cdot{z}^{j+1}.
	\end{equation*}
\end{proof}

\begin{corollary}
\label{cor_prod_Sk_T}
Let $\alpha$, $\beta$ and $\gamma$ be nonconstant inner functions and let $\varphi,\psi\in H^2$.
\begin{enumerate}[(a)]
  \item Assume that $A_{\varphi}^{\beta,\gamma}\in\mathcal{T}(\beta,\gamma)$ and $U_{\psi}^{\alpha,\beta}\in\mathcal{S}_k(\alpha,\beta)$. If $\gamma\leq\beta$ and $W_k^*\beta\leq\alpha$, then $U=A_{\varphi}^{\beta,\gamma}U_{\psi}^{\alpha,\beta}=U_{\eta}^{\alpha,\gamma}\in\mathcal{S}_{k}(\alpha,\gamma)$ with
  $\displaystyle\eta=\sum_{j=0}^{k-1}W_{k}^*W_{k}P_{W_{k}^*\gamma}\big((W_{k}^*\varphi)P_{W_{k}^*\beta}(\psi z^j)\big)\cdot \overline{z}^j$.
  \item Assume that $U_{\varphi}^{\beta,\gamma}\in\mathcal{S}_m(\beta,\gamma)$ and $A_{\psi}^{\alpha,\beta}\in\mathcal{T}(\alpha,\beta)$. If $W_m^*\gamma\leq\beta\leq\alpha$, then $U=U_{\varphi}^{\beta,\gamma}A_{\psi}^{\alpha,\beta}=U_{\eta}^{\alpha,\gamma}\in\mathcal{S}_{m}(\alpha,\gamma)$ with
 $\displaystyle\eta=\sum_{j=0}^{m-1}W_{m}^*W_{m}P_{W_{m}^*\gamma}\big(\varphi P_{\beta}(\psi z^j)\big)\cdot \overline{z}^j$.

\item Assume that $A_{\overline{\varphi}}^{\beta,\gamma}\in\mathcal{T}(\beta,\gamma)$ and $U_{\overline{\psi}}^{\alpha,\beta}\in\mathcal{S}_k(\alpha,\beta)$. If $\alpha\leq W_k^*\beta$ and $\beta\leq \gamma$, then $U=A_{\overline{\varphi}}^{\beta,\gamma}U_{\overline{\psi}}^{\alpha,\beta}=U_{\zeta}^{\alpha,\gamma}\in\mathcal{S}_{k}(\alpha,\gamma)$ with $\displaystyle\zeta=\overline{\alpha}\sum_{j=0}^{k-1}W_{k}^*W_{k}P_{W_{k}^*\gamma}\big(\overline{(W_{k}^*\varphi)}\overline{\psi}P_{\alpha}(\alpha\overline{z}^{j+1})\big)\cdot z^{j+1}$.
\item Assume that $U_{\overline{\varphi}}^{\beta,\gamma}\in\mathcal{S}_m(\beta,\gamma)$ and $A_{\overline{\psi}}^{\alpha,\beta}\in\mathcal{T}(\alpha,\beta)$. If $\alpha\leq \beta\leq W_m^*\gamma$, then $U=U_{\overline{\varphi}}^{\beta,\gamma}A_{\overline{\psi}}^{\alpha,\beta}=U_{\zeta}^{\alpha,\gamma}\in\mathcal{S}_{m}(\alpha,\gamma)$ with $\displaystyle\zeta=\overline{\alpha}\sum_{j=0}^{m-1}W_{m}^*W_{m}P_{W_{m}^*\gamma}\big(\overline{\varphi}\overline{\psi}P_{\alpha}(\alpha\overline{z}^{j+1})\big)\cdot z^{j+1}$.
%
%
\end{enumerate}
\end{corollary}

	Note that under the assumptions from Proposition \ref{thm_product_S_km}(a) for $U=U_{\varphi}^{\beta,\gamma}U_{\psi}^{\alpha,\beta}$ we also have
	\begin{align*}
	U -S^*_{\gamma}US_{\alpha}^{km}&=U_{\varphi}^{\beta,\gamma}U_{\psi}^{\alpha,\beta} -S^*_{\gamma}U_{\varphi}^{\beta,\gamma}U_{\psi}^{\alpha,\beta}S_{\alpha}^{km}\\
	&=(I-S^*_{\gamma}S_{\gamma})U_{\varphi}^{\beta,\gamma}U_{\psi}^{\alpha,\beta}
	=(\widetilde{k}_0^\gamma\otimes \widetilde{k}_{0}^\gamma)U\\
	&=\widetilde{k}_0^\gamma\otimes U^*\widetilde{k}_{0}^\gamma,
	\end{align*}
which implies that $U=U_{\varphi}^{\beta,\gamma}U_{\psi}^{\alpha,\beta}=U_{\eta}^{\alpha,\gamma}\in\mathcal{S}_{km}(\alpha,\gamma)$ with
$$\eta=(W^*_{km}\gamma)\overline{z}^{km}\overline{U^*\widetilde{k}_{0}^\gamma}.$$
Now for $f\in K_\alpha^\infty$ we have
\begin{align*}
	\langle U^*\widetilde{k}_{0}^\gamma,f\rangle&=\langle (U_{\psi}^{\alpha,\beta})^*(U_{\varphi}^{\beta,\gamma})^*\widetilde{k}_{0}^\gamma,f\rangle=\langle (U_{\varphi}^{\beta,\gamma})^*\widetilde{k}_{0}^\gamma,U_{\psi}^{\alpha,\beta}f\rangle\\
	&=\langle P_{\beta}(\overline{\varphi}W_m^*\widetilde{k}_{0}^\gamma),P_{\beta}W_k(\psi f)\rangle=\langle W_k^*P_{\beta}(\overline{\varphi}W_m^*\widetilde{k}_{0}^\gamma),\psi f\rangle\\
	&=\langle \overline{\psi}P_{W_k^*\beta}W_k^*(\overline{\varphi}W_m^*\widetilde{k}_{0}^\gamma), f\rangle=\langle P_\alpha\big(\overline{\psi}P_{W_k^*\beta}(\overline{(W_k^*\varphi)}\cdot W_{km}^*\widetilde{k}_{0}^\gamma)\big), f\rangle\\
	&=\langle P_\alpha\big(\overline{\psi}P_{W_k^*\beta}(\overline{(W_k^*\varphi)}\cdot \overline{z}^{km}(W_{km}^*\gamma-\gamma(0)))\big), f\rangle\\
	&=\langle P_\alpha\big(\overline{\psi}P_{W_k^*\beta}(\overline{(W_k^*\varphi)}\cdot \overline{z}^{km}W_{km}^*\gamma)\big), f\rangle=\langle P\big(\overline{\psi}P_{W_k^*\beta}(\overline{(W_k^*\varphi)}\cdot \overline{z}^{km}W_{km}^*\gamma)\big), f\rangle.
	\end{align*}
The last equality follows from the fact that $\psi\in H^2$ and $W_k^*\beta\leq\alpha$. Thus
$$\eta=(W^*_{km}\gamma)\overline{z}^{km}\overline{P\big(\overline{\psi}P_{W_k^*\beta}(\overline{(W_k^*\varphi)}\cdot \overline{z}^{km}W_{km}^*\gamma))}.$$

	Analogously, under assumptions from part (b) of Proposition \ref{thm_product_S_km}, for $U=U_{\overline{\varphi}}^{\beta,\gamma}U_{\overline{\psi}}^{\alpha,\beta}$ we have
	\begin{align*}
	U -S_{\gamma}U(S^*_{\alpha})^{km}
	={k}_0^\gamma\otimes U^*{k}_{0}^\gamma
	\end{align*}
and so $U=U_{\zeta}^{\alpha,\gamma}\in\mathcal{S}_{km}(\alpha,\gamma)$ with
$$\zeta=\overline{U^*{k}_{0}^\gamma}.$$
Here for $f\in K_\alpha^\infty$ we have
\begin{align*}
	\langle U^*{k}_{0}^\gamma,f\rangle&=\langle (U_{\overline{\psi}}^{\alpha,\beta})^*(U_{\overline{\varphi}}^{\beta,\gamma})^*{k}_{0}^\gamma,f\rangle=\langle (U_{\overline{\varphi}}^{\beta,\gamma})^*{k}_{0}^\gamma,U_{\overline{\psi}}^{\alpha,\beta}f\rangle\\
	&=\langle P_{\beta}({\varphi}W_m^*{k}_{0}^\gamma),P_{\beta}W_k(\overline{\psi} f)\rangle=\langle {\varphi}W_m^*{k}_{0}^\gamma,W_kP_{W_k^*\beta}(\overline{\psi} f)\rangle\\
	&=\langle W_k^*({\varphi}W_m^*{k}_{0}^\gamma),P_{W_k^*\beta}(\overline{\psi} f)\rangle=\langle (W_k^*{\varphi})(W_{km}^*{k}_{0}^\gamma),P(\overline{\psi} f)\rangle\\
	&=\langle \psi(W_k^*{\varphi})(W_{km}^*{k}_{0}^\gamma), f\rangle=\langle P_{\alpha}\big( \psi(W_k^*{\varphi})(1-\overline{\gamma(0)}W_{km}^*\gamma)\big), f\rangle\\
	&=\langle P_{\alpha}\big( \psi(W_k^*{\varphi})\big), f\rangle
\end{align*}
This time we used the fact that $\alpha\leq W_k^*\beta$ (line 3) and $\alpha\leq W_k^*\beta\leq W_{km}^*\gamma$ (the last equality). Thus
$$\zeta=\overline{P_{\alpha}\big( \psi(W_k^*{\varphi})\big)}.$$

%

\begin{corollary}
	Let $\alpha$, $\beta$ and $\gamma$ be nonconstant inner functions and let $\varphi,\psi\in H^2$.
	\begin{enumerate}
		\item[(a)] Assume that $U_{\varphi}^{\beta,\gamma}\in\mathcal{S}_m(\beta,\gamma)$ and $U_{\psi}^{\alpha,\beta}\in\mathcal{S}_k(\alpha,\beta)$. If $W_m^*\gamma\leq\beta$ and $W_k^*\beta\leq\alpha$, then $U=U_{\varphi}^{\beta,\gamma}U_{\psi}^{\alpha,\beta}\in\mathcal{S}_{km}(\alpha,\gamma)$ and $U=U_{\eta}^{\alpha,\gamma}$ with
	$$\eta=(W^*_{km}\gamma)\overline{z}^{km}\overline{P\big(\overline{\psi}P_{W_k^*\beta}(\overline{(W_k^*\varphi)}\cdot \overline{z}^{km}W_{km}^*\gamma))}.$$
		\item[(b)] Assume that $U_{\overline{\varphi}}^{\beta,\gamma}\in\mathcal{S}_m(\beta,\gamma)$ and $U_{\overline{\psi}}^{\alpha,\beta}\in\mathcal{S}_k(\alpha,\beta)$. If $\alpha\leq W_k^*\beta$ and $\beta\leq W_m^*\gamma$, then $U=U_{\overline{\varphi}}^{\beta,\gamma}U_{\overline{\psi}}^{\alpha,\beta}\in\mathcal{S}_{km}(\alpha,\gamma)$ and $U=U_{\zeta}^{\alpha,\gamma}$ with
		$$\zeta=\overline{P_{\alpha}\big( \psi(W_k^*{\varphi})\big)}.$$
		
	\end{enumerate}
\end{corollary}

\begin{corollary}
	Let $\alpha$, $\beta$ and $\gamma$ be nonconstant inner functions and let $\varphi,\psi\in H^2$.
	\begin{enumerate}
		\item[(a)] Assume that $A_{\varphi}^{\beta,\gamma}\in\mathcal{T}(\beta,\gamma)$ and $U_{\psi}^{\alpha,\beta}\in\mathcal{S}_k(\alpha,\beta)$. If $\gamma\leq\beta$ and $W_k^*\beta\leq\alpha$, then $U=A_{\varphi}^{\beta,\gamma}U_{\psi}^{\alpha,\beta}=U_{\eta}^{\alpha,\gamma}\in\mathcal{S}_{k}(\alpha,\gamma)$ with $\eta=(W^*_{k}\gamma)\overline{z}^{k}\overline{P\big(\overline{\psi}P_{W_k^*\beta}(\overline{(W_k^*\varphi)}\cdot \overline{z}^{k}W_{k}^*\gamma))}$.
		\item[(b)] Assume that $U_{\varphi}^{\beta,\gamma}\in\mathcal{S}_m(\beta,\gamma)$ and $A_{\psi}^{\alpha,\beta}\in\mathcal{T}(\alpha,\beta)$. If $W_m^*\gamma\leq\beta\leq\alpha$, then $U=U_{\varphi}^{\beta,\gamma}A_{\psi}^{\alpha,\beta}=U_{\eta}^{\alpha,\gamma}\in\mathcal{S}_{m}(\alpha,\gamma)$ with
		$\eta=(W^*_{m}\gamma)\overline{z}^{m}\overline{P\big(\overline{\psi}P_{\beta}(\overline{\varphi}\cdot \overline{z}^{m}W_{m}^*\gamma))}$.
		
		\item[(c)] Assume that $A_{\overline{\varphi}}^{\beta,\gamma}\in\mathcal{T}(\beta,\gamma)$ and $U_{\overline{\psi}}^{\alpha,\beta}\in\mathcal{S}_k(\alpha,\beta)$. If $\alpha\leq W_k^*\beta$ and $\beta\leq \gamma$, then $U=A_{\overline{\varphi}}^{\beta,\gamma}U_{\overline{\psi}}^{\alpha,\beta}=U_{\zeta}^{\alpha,\gamma}\in\mathcal{S}_{k}(\alpha,\gamma)$ with
		$\zeta=\overline{P_{\alpha}\big( \psi(W_k^*{\varphi})\big)}$.
		\item[(d)] Assume that $U_{\overline{\varphi}}^{\beta,\gamma}\in\mathcal{S}_m(\beta,\gamma)$ and $A_{\overline{\psi}}^{\alpha,\beta}\in\mathcal{T}(\alpha,\beta)$. If $\alpha\leq\beta\leq W_m^*\gamma$, then $U=U_{\overline{\varphi}}^{\beta,\gamma}U_{\overline{\psi}}^{\alpha,\beta}=U_{\zeta}^{\alpha,\gamma}\in\mathcal{S}_{m}(\alpha,\gamma)$ with
		$\zeta=\overline{P_{\alpha}\big( \psi{\varphi}\big)}$.
		
	\end{enumerate}
\end{corollary}

Finally, we can consider the case when $k=m=1$ (see \cite[Proposition 1]{Y}).
\begin{corollary}
	Let $\alpha$, $\beta$ and $\gamma$ be nonconstant inner functions and let $\varphi,\psi\in H^2$.
	\begin{enumerate}[(a)]
		\item Assume that $A_{\varphi}^{\beta,\gamma}\in\mathcal{T}(\beta,\gamma)$ and $A_{\psi}^{\alpha,\beta}\in\mathcal{T}(\alpha,\beta)$. If $\gamma\leq\beta\leq\alpha$, then $A=A_{\varphi}^{\beta,\gamma}A_{\psi}^{\alpha,\beta}=A_{\eta}^{\alpha,\gamma}\in\mathcal{T}(\alpha,\gamma)$ with
		$\displaystyle\eta=P_{\gamma}\big(\varphi P_{\beta}(\psi)\big)$.
		
		\item Assume that $A_{\varphi}^{\beta,\gamma}\in\mathcal{T}(\beta,\gamma)$ and $A_{\psi}^{\alpha,\beta}\in\mathcal{T}(\alpha,\beta)$. If $\gamma\leq\beta\leq\alpha$, then $A=A_{\varphi}^{\beta,\gamma}A_{\psi}^{\alpha,\beta}=A_{\eta}^{\alpha,\gamma}\in\mathcal{T}(\alpha,\gamma)$ with
		$\eta=\gamma\overline{z}\overline{P\big(\overline{\psi}P_{\beta}(\gamma\overline{z}\overline{\varphi})\big)}$.

		\item Assume that $A_{\overline{\varphi}}^{\beta,\gamma}\in\mathcal{T}(\beta,\gamma)$ and $A_{\overline{\psi}}^{\alpha,\beta}\in\mathcal{T}(\alpha,\beta)$. If $\alpha\leq \beta\leq \gamma$, then $A=A_{\overline{\varphi}}^{\beta,\gamma}A_{\overline{\psi}}^{\alpha,\beta}=A_{\zeta}^{\alpha,\gamma}\in\mathcal{T}(\alpha,\gamma)$ with $\displaystyle\zeta=\overline{\alpha}zP_{\gamma}\big(\overline{\varphi}\overline{\psi}P_{\alpha}(\alpha\overline{z})\big)$.

		\item Assume that $A_{\overline{\varphi}}^{\beta,\gamma}\in\mathcal{T}(\beta,\gamma)$ and $A_{\overline{\psi}}^{\alpha,\beta}\in\mathcal{T}(\alpha,\beta)$. If $\alpha\leq \beta\leq \gamma$, then $A=A_{\overline{\varphi}}^{\beta,\gamma}A_{\overline{\psi}}^{\alpha,\beta}=A_{\zeta}^{\alpha,\gamma}\in\mathcal{T}(\alpha,\gamma)$ with
		$\zeta=\overline{P_{\alpha}\big( \psi{\varphi}\big)}$.
		
	\end{enumerate}
\end{corollary}

\section{Products of operators from $\mathcal{S}_k(\alpha,\beta)$ with $L^2$ symbols}

Let $\alpha$, $\beta$ be two nonconstant inner functions and fix $k\in\mathbb{N}$ such that $k\leq \dim K_\alpha$. Recall that in that case $U_{\varphi}^{\alpha,\beta}=0$ if and only if $\varphi\in\overline{\alpha H^2}+\overline{z}^{k-1}(W_k^*\beta)H^2$ (see \cite{BM2}). It follows that every operator from $\mathcal{S}_k(\alpha,\beta)$ has a symbol of the form $\displaystyle\overline{\varphi}_-+\varphi_+$, where $\varphi_-\in K_\alpha$ and $\varphi_+=\sum\limits_{j=1}^{k}z^j(W_k^*\varphi_j)\in z K_{W_k^*\beta}$ for some $\varphi_j\in K_\beta$, $1\leq j\leq k$. To see this note that
  $$L^2=\overline{H^2}\oplus zH^2 =\overline{\alpha H^2}\oplus \overline{K_{\alpha}}\oplus z(W_k^*\beta) H^2\oplus zK_{W_k^*\beta}. $$
  Decomposing an arbitrary symbol $\varphi\in L^2$ accordingly as
  $$\varphi=\overline{\alpha h}+\overline{\varphi}_-+(W_k^*\beta)zg+z\chi\quad (h,g\in H^2, \varphi_-\in K_\alpha,\chi\in K_{W_k^*\beta})$$
  we get
  $$U_{\varphi}^{\alpha,\beta}=U_{\overline{\varphi}_-+z\chi}^{\alpha,\beta}$$
  since
  $$z(W_k^*\beta)H^2=\overline{z}^{k-1}(W_k^*\beta)z^kH^2\subset \overline{z}^{k-1}(W_k^*\beta)H^2.$$

Now by the decomposition
$$K_{W_k^*\beta}=W_k^*(K_\beta)\oplus zW_k^*(K_\beta)\oplus\ldots \oplus z^{k-1}W_k^*(K_\beta)$$
(see formula (4.3) from \cite{BM2}), there are functions $\varphi_1,\ldots,\varphi_k\in K_\beta$ such that $\displaystyle\chi=\sum_{j=1}^kz^{j-1}(W_k^*\varphi_j)$ and so
\begin{equation}\label{iks}
\overline{\varphi}_-+{\varphi}_+=\overline{\varphi}_-+z\chi=\overline{\varphi}_-+\sum_{j=1}^kz^{j}(W_k^*\varphi_j).
\end{equation}
Observe that this decomposition is orthogonal.
Moreover, it can be shown that if $U\in \mathcal{S}_k(\alpha,\beta)$ has a symbol given by \eqref{iks}, then
\begin{equation*}
\begin{split}
 U-S_\beta U(S^*_{\alpha})^k
  &=k_0^\beta\otimes \varphi_-+\sum_{j=0}^{k-1} \tfrac{1}{j!}(S_\beta P_\beta W_k(\bar{z}^{k-j}\varphi_+))\otimes k_{0,j}^\alpha\\
  &=k_0^\beta\otimes \varphi_-+\sum_{j=0}^{k-1} \tfrac{1}{j!}(S_\beta \varphi_{k-j})\otimes k_{0,j}^\alpha.
\end{split}
\end{equation*}
See the proof of Theorem 2 in \cite{BM2} for more detailed computations.

\begin{theorem}
\label{thm_prod_L2-symb_Sk_UU}
Let $\alpha$, $\beta$ and $\gamma$ be nonconstant inner functions and let $k$, $m$ be two fixed positive integers such that $k\leq\dim K_\alpha$ and $m\leq \dim K_\beta$.
Assume that $A=U_{\overline{\varphi}_-+\varphi_+}^{\beta,\gamma}\in\mathcal{S}_m(\beta,\gamma)$ for some $\varphi_-\in K_\beta$, $\varphi_+\in zK_{W^*_m\gamma}$  and $B=U_{\overline{\psi}_-+\psi_+}^{\alpha,\beta}\in\mathcal{S}_k(\alpha,\beta)$ for some $\psi_-\in K_\alpha$, $\psi_+\in zK_{W^*_k\beta}$. Then $AB\in\mathcal{S}_{km}(\alpha,\gamma)$
if and only if there exist
$\tau_p\in K_\gamma$, $p=0,1,\ldots,mk-1$, and $\upsilon\in K_\alpha$ such that

\begin{align*}
	\sum_{n=0}^{m-1}\tfrac1{n!}(Ak_{0,n}^\beta)&\otimes(S_\alpha^{kn}\psi_-)-\sum_{j=0}^{m-1}\left(\tfrac{1}{j!}\right)^2\!\!\cdot( S_\gamma A\widetilde{k}_{0,j}^\beta)\otimes (S_\alpha^{km}B^*\widetilde{k}_{0,j}^\beta)\\&+\sum_{j=0}^{m-1} \tfrac{1}{j!}(S_\gamma P_\gamma W_m(\bar{z}^{m-j}\varphi_+))\otimes \left(B^*k_{0,j}^\beta-\sum_{n=0}^{m-1}  {\langle S_\beta^{n}k_0^\beta,k_{0,j}^\beta\rangle}S_\alpha^{kn}\psi_-\right)
	\\
	= k_0^\gamma&\otimes \upsilon+\sum_{p=0}^{km-1} \tau_p\otimes k_{0,p}^\alpha.
\end{align*}
In that case $AB=U_{\xi}^{\alpha,\gamma}$ with
\begin{align*}
   \xi=\overline{\upsilon}&+\overline{B^*\psi_-}-\sum_{n=0}^{m-1}{\langle S_\beta^{n}k_0^\beta,\varphi_-\rangle} \overline{S_\alpha^{kn}\psi_-}-\sum_{n=0}^{m-1}\sum_{l=0}^{k-1} \frac{1}{l!}{\langle S_\beta^{n+1} P_\beta W_k(\overline{z}^{k-l}\psi_+),\varphi_-\rangle} S_\alpha^{kn}\overline{k_{0,l}^\alpha}\\
   &+\sum_{p=0}^{km-1} (W_{km}^*\tau_{p})p!\overline{z}^p+\sum_{p=0}^{km-1}\sum_{n=0}^{m-1}(W_{km}^*AS_\beta^{n+1} P_\beta W_k(\overline{z}^{k(n+1)-p}\psi_+))\overline{z}^p\\&-\sum_{p=0}^{km-1}\sum_{n=0}^{m-1}\left(\sum_{j=0}^{m-1} \tfrac{1}{j!}\langle S_\beta^{n+1} P_\beta W_k(\overline{z}^{k(n+1)-p}\psi_+),k_{0,j}^\beta\rangle S_\gamma P_\gamma W_m(\bar{z}^{m-j}\varphi_+)\right)\overline{z}^p.
\end{align*}
\end{theorem}

\begin{proof}
	Let $A=U_{\overline{\varphi}_-+\varphi_+}^{\beta,\gamma}\in\mathcal{S}_m(\beta,\gamma)$ with $\varphi_-\in K_\beta$, $\varphi_+\in zK_{W^*_m\gamma}$  and $B=U_{\overline{\psi}_-+\psi_+}^{\alpha,\beta}\in\mathcal{S}_k(\alpha,\beta)$ with $\psi_-\in K_\alpha$, $\psi_+\in zK_{W^*_k\beta}$. As mentioned above, we then have
	\begin{equation}
	\label{A}
	A-S_\gamma A(S^*_{\beta})^m
	=k_0^\gamma\otimes \varphi_-+\sum_{j=0}^{m-1} \tfrac{1}{j!}(S_\gamma P_\gamma W_m(\bar{z}^{m-j}\varphi_+))\otimes k_{0,j}^\beta
	\end{equation}
	and
	\begin{equation}
	\label{B}
	B-S_\beta B(S^*_{\alpha})^k
	=k_0^\beta\otimes \psi_-+\sum_{l=0}^{k-1} \tfrac{1}{l!}(S_\beta P_\beta W_k(\bar{z}^{k-l}\psi_+))\otimes k_{0,l}^\alpha.
	\end{equation}
	By \eqref{B}, for any nonnegative integer $n$ we have
	\begin{align*}
	S_\beta^{n} B(S^*_\alpha)^{kn}-&S_\beta^{n+1}B(S^*_\alpha)^{(n+1)k}\\&=(S_\beta^{n}k_0^\beta)\otimes(S_\alpha^{kn}\psi_-)+\sum_{l=0}^{k-1} \frac{1}{l!}S_\beta^{n+1}P_\beta W_k(\overline{z}^{k-l}\psi_+)\otimes S_\alpha^{kn}k_{0,l}^\alpha.
	\end{align*}
	Adding the above for $n=0,1,2,\ldots,m-1$, we obtain
	\begin{align}\label{C}
	B-S_\beta^{m} B(S^*_\alpha)^{km}&
	=\sum_{n=0}^{m-1}(S_\beta^{n}k_0^\beta)\otimes(S_\alpha^{kn}\psi_-)
	+\sum_{n=0}^{m-1}\sum_{l=0}^{k-1} \frac{1}{l!}S_\beta^{n+1} P_\beta W_k(\overline{z}^{k-l}\psi_+)\otimes S_\alpha^{kn}k_{0,l}^\alpha.
	\end{align}
	
Now the product $AB$ belongs to $\mathcal{S}_{km}(\alpha,\gamma)$ if and only if there exist $\Psi\in K_\alpha$ and $\Phi_j\in K_\gamma$, $j=0,1,\ldots,km-1$, such that
\begin{equation}
\label{eq_aux_thm_prod_UU_1}
  AB-S_\gamma AB(S^*_\alpha)^{km}=k_0^\gamma\otimes\Psi+\sum_{j=0}^{km-1} \Phi_j\otimes k_{0,j}^\alpha.
\end{equation}
By Lemma \ref{lem_sar_k}, we have
\begin{displaymath}
\begin{split}
S_\gamma AB(S^*_\alpha)^{km}&=S_\gamma A(S^*_\beta)^{m}S_\beta^{m}B(S^*_\alpha)^{km}+S_\gamma A(I_{K_\beta}-(S^*_\beta)^{m}S_\beta^{m})B(S^*_\alpha)^{km}\\
&=S_\gamma A(S^*_\beta)^{m}S_\beta^{m}B(S^*_\alpha)^{km}+S_\gamma A\left(\sum_{j=0}^{m-1}\left(\tfrac{1}{j!}\right)^2\!\!\cdot( \widetilde{k}_{0,j}^\beta\otimes \widetilde{k}_{0,j}^\beta)\right)B(S^*_\alpha)^{km}\\
&=S_\gamma A(S^*_\beta)^{m}S_\beta^{m}B(S^*_\alpha)^{km}+\sum_{j=0}^{m-1}\left(\tfrac{1}{j!}\right)^2\!\!\cdot( S_\gamma A\widetilde{k}_{0,j}^\beta)\otimes (S_\alpha^{km}B^*\widetilde{k}_{0,j}^\beta).\\
\end{split}
\end{displaymath}
Using \eqref{A} and \eqref{C} we obtain
\begin{align*}
S_\gamma A(S^*_\beta)^{m}S_\beta^{m}B&(S^*_\alpha)^{km}\\
=AB&-\sum_{n=0}^{m-1}(AS_\beta^{n}k_0^\beta)\otimes(S_\alpha^{kn}\psi_-)
-\sum_{n=0}^{m-1}\sum_{l=0}^{k-1} \frac{1}{l!}AS_\beta^{n+1} P_\beta W_k(\overline{z}^{k-l}\psi_+)\otimes S_\alpha^{kn}k_{0,l}^\alpha\\
&-k_0^\gamma\otimes B^*\varphi_--\sum_{j=0}^{m-1} \tfrac{1}{j!}(S_\gamma P_\gamma W_m(\bar{z}^{m-j}\varphi_+))\otimes B^*k_{0,j}^\beta\\
&+\sum_{n=0}^{m-1}\langle S_\beta^{n}k_0^\beta,\varphi_-\rangle k_0^\gamma\otimes(S_\alpha^{kn}\psi_-)\\
&+\sum_{n=0}^{m-1}\sum_{l=0}^{k-1} \frac{1}{l!}\langle S_\beta^{n+1} P_\beta W_k(\overline{z}^{k-l}\psi_+),\varphi_-\rangle k_0^\gamma\otimes S_\alpha^{kn}k_{0,l}^\alpha\\
&+\sum_{j=0}^{m-1} \sum_{n=0}^{m-1}  \tfrac{\langle S_\beta^{n}k_0^\beta,k_{0,j}^\beta\rangle}{j!}(S_\gamma P_\gamma W_m(\bar{z}^{m-j}\varphi_+))\otimes(S_\alpha^{kn}\psi_-)\\
&+
\sum_{n=0}^{m-1}\sum_{l=0}^{k-1}\sum_{j=0}^{m-1} \tfrac{\langle S_\beta^{n+1} P_\beta W_k(\overline{z}^{k-l}\psi_+),k_{0,j}^\beta\rangle}{j! l!}(S_\gamma P_\gamma W_m(\bar{z}^{m-j}\varphi_+))\otimes S_\alpha^{kn}k_{0,l}^\alpha.
\end{align*}
Note that $$S_\alpha^{kn}k_{0,l}^\alpha=\tfrac{l!}{(kn+l)!}k_{0,kn+l}^\alpha.$$
It follows that
\begin{align*}
AB-\ S_\gamma &AB(S^*_\alpha)^{km}\\
=k_0^\gamma\otimes&\left(B^*\psi_--\sum_{n=0}^{m-1}\overline{\langle S_\beta^{n}k_0^\beta,\varphi_-\rangle} S_\alpha^{kn}\psi_--\sum_{n=0}^{m-1}\sum_{l=0}^{k-1} \frac{1}{l!}\overline{\langle S_\beta^{n+1} P_\beta W_k(\overline{z}^{k-l}\psi_+),\varphi_-\rangle} S_\alpha^{kn}k_{0,l}^\alpha\right)\\
+\sum_{n=0}^{m-1}&\sum_{l=0}^{k-1}\left(\tfrac{1}{(kn+l)!}AS_\beta^{n+1} P_\beta W_k(\overline{z}^{k-l}\psi_+)  \right.\\
&\left.-\sum_{j=0}^{m-1} \tfrac{\langle S_\beta^{n+1} P_\beta W_k(\overline{z}^{k-l}\psi_+),k_{0,j}^\beta\rangle}{j!(kn+l)!}S_\gamma P_\gamma W_m(\bar{z}^{m-j}\varphi_+)\right)\otimes k_{0,kn+l}^\alpha\\
+\sum_{n=0}^{m-1}&(AS_\beta^{n}k_0^\beta)\otimes(S_\alpha^{kn}\psi_-)+\sum_{j=0}^{m-1} \tfrac{1}{j!}(S_\gamma P_\gamma W_m(\bar{z}^{m-j}\varphi_+))\otimes B^*k_{0,j}^\beta\\
-\sum_{j=0}^{m-1}& \sum_{n=0}^{m-1}  \tfrac{\langle S_\beta^{n}k_0^\beta,k_{0,j}^\beta\rangle}{j!}(S_\gamma P_\gamma W_m(\bar{z}^{m-j}\varphi_+))\otimes(S_\alpha^{kn}\psi_-)-\sum_{j=0}^{m-1}\left(\tfrac{1}{j!}\right)^2\!\!\cdot( S_\gamma A\widetilde{k}_{0,j}^\beta)\otimes (S_\alpha^{km}B^*\widetilde{k}_{0,j}^\beta)
\end{align*}
and so $AB$ satisfies \eqref{eq_aux_thm_prod_UU_1} if and only if there exist functions
$\upsilon\in K_\alpha$ and $\tau_p\in K_\gamma$, $p=0,1,\ldots,mk-1$, such that
\begin{align*}
 \sum_{n=0}^{m-1}\tfrac1{n!}(Ak_{0,n}^\beta)\otimes(S_\alpha^{kn}\psi_-)&-\sum_{j=0}^{m-1}\left(\tfrac{1}{j!}\right)^2\!\!\cdot( S_\gamma A\widetilde{k}_{0,j}^\beta)\otimes (S_\alpha^{km}B^*\widetilde{k}_{0,j}^\beta)\\&+\sum_{j=0}^{m-1} \tfrac{1}{j!}(S_\gamma P_\gamma W_m(\bar{z}^{m-j}\varphi_+))\otimes \left(B^*k_{0,j}^\beta-\sum_{n=0}^{m-1}  {\langle S_\beta^{n}k_0^\beta,k_{0,j}^\beta\rangle}S_\alpha^{kn}\psi_-\right)\\
 \\
 = k_0^\gamma\otimes \upsilon&+\sum_{p=0}^{km-1} \tau_p\otimes k_{0,p}^\alpha.
 \end{align*}

Formula for the symbol of $AB$ follows from \eqref{ox}.
\end{proof}


\begin{corollary}
	Let $\alpha$, $\beta$ and $\gamma$ be nonconstant inner functions and let $k$ be a fixed positive integer such that $k\leq\dim K_\alpha$.
	Assume that $A=A_{\overline{\varphi}_-+\varphi_+}^{\beta,\gamma}\in\mathcal{T}(\beta,\gamma)$ for some $\varphi_-\in K_\beta$, $\varphi_+\in zK_{\gamma}$  and $B=U_{\overline{\psi}_-+\psi_+}^{\alpha,\beta}\in\mathcal{S}_k(\alpha,\beta)$ for some $\psi_-\in K_\alpha$, $\psi_+\in zK_{W^*_k\beta}$. Then $AB\in\mathcal{S}_{k}(\alpha,\gamma)$
	if and only if there exist
	$\tau_p\in K_\gamma$, $p=0,1,\ldots,k-1$, and $\upsilon\in K_\alpha$ such that
	
	\begin{align*}
		Ak_{0}^\beta\otimes \psi_-- S_\gamma A\widetilde{k}_{0}^\beta\otimes S_\alpha^{k}B^*\widetilde{k}_{0}^\beta+(S_\gamma P_\gamma (\bar{z}\varphi_+))&\otimes (B^*k_{0}^\beta-{\langle k_0^\beta,k_{0}^\beta\rangle}\psi_-)	\\
			= k_0^\gamma&\otimes \upsilon+\sum_{p=0}^{k-1} \tau_p\otimes k_{0,p}^\alpha.
	\end{align*}
	In that  case $AB=U_{\xi}^{\alpha,\gamma}$ with
	\begin{align*}
		\xi=\overline{\upsilon}&+\overline{B^*\psi_-}-{\langle  k_0^\beta,\varphi_-\rangle} \overline{\psi_-}-\sum_{l=0}^{k-1} \frac{1}{l!}{\langle S_\beta P_\beta W_k(\overline{z}^{k-l}\psi_+),\varphi_-\rangle} \overline{k_{0,l}^\alpha}\\
		&+\sum_{p=0}^{k-1} (W_{k}^*\tau_{p})p!\overline{z}^p+\sum_{p=0}^{k-1}(W_{k}^*AS_\beta P_\beta W_k(\overline{z}^{k}\psi_+))\overline{z}^p\\&-\sum_{p=0}^{k-1}\langle  P_\beta W_k(\overline{z}^{k-p}\psi_+),k_{0}^\beta\rangle S_\gamma P_\gamma (\bar{z}\varphi_+)\overline{z}^p.
	\end{align*}
\end{corollary}

%
%
%

\begin{corollary}
	Let $\alpha$, $\beta$ and $\gamma$ be nonconstant inner functions and let $m$ be a fixed positive integer such that $m\leq \dim K_\beta$.
	Assume that $A=U_{\overline{\varphi}_-+\varphi_+}^{\beta,\gamma}\in\mathcal{S}_m(\beta,\gamma)$ for some $\varphi_-\in K_\beta$, $\varphi_+\in zK_{W^*_m\gamma}$  and $B=A_{\overline{\psi}_-+\psi_+}^{\alpha,\beta}\in\mathcal{T}(\alpha,\beta)$ for some $\psi_-\in K_\alpha$, $\psi_+\in zK_{\beta}$. Then $AB\in\mathcal{S}_{m}(\alpha,\gamma)$
	if and only if there exist
	$\tau_p\in K_\gamma$, $p=0,1,\ldots,m-1$, and $\upsilon\in K_\alpha$ such that
	
	\begin{align*}
		\sum_{n=0}^{m-1}\tfrac1{n!}(Ak_{0,n}^\beta)\otimes(S_\alpha^{n}\psi_-)&-\sum_{j=0}^{m-1}\left(\tfrac{1}{j!}\right)^2\!\!\cdot( S_\gamma A\widetilde{k}_{0,j}^\beta)\otimes (S_\alpha^{m}B^*\widetilde{k}_{0,j}^\beta)\\&+\sum_{j=0}^{m-1} \tfrac{1}{j!}(S_\gamma P_\gamma W_m(\bar{z}^{m-j}\varphi_+))\otimes \left(B^*k_{0,j}^\beta-\sum_{n=0}^{m-1}  {\langle S_\beta^{n}k_0^\beta,k_{0,j}^\beta\rangle}S_\alpha^{n}\psi_-\right)
		\\
		= k_0^\gamma\otimes \upsilon&+\sum_{p=0}^{m-1} \tau_p\otimes k_{0,p}^\alpha.
	\end{align*}
	In that case $AB=U_{\xi}^{\alpha,\gamma}$ with
	\begin{align*}
		\xi=\overline{\upsilon}&+\overline{B^*\psi_-}-\sum_{n=0}^{m-1}{\langle S_\beta^{n}k_0^\beta,\varphi_-\rangle} \overline{S_\alpha^{n}\psi_-}-\sum_{n=0}^{m-1} {\langle S_\beta^{n+1} P_\beta W_k(\overline{z}\psi_+),\varphi_-\rangle} S_\alpha^{n}\overline{k_{0}^\alpha}\\
		&+\sum_{p=0}^{m-1} (W_{m}^*\tau_{p})p!\overline{z}^p+\sum_{p=0}^{m-1}\sum_{n=0}^{m-1}(W_{m}^*AS_\beta^{n+1} P_\beta (\overline{z}^{n+1-p}\psi_+))\overline{z}^p\\&-\sum_{p=0}^{m-1}\sum_{n=0}^{m-1}\left(\sum_{j=0}^{m-1} \tfrac{1}{j!}\langle S_\beta^{n+1} P_\beta (\overline{z}^{n+1-p}\psi_+),k_{0,j}^\beta\rangle S_\gamma P_\gamma W_m(\bar{z}^{m-j}\varphi_+)\right)\overline{z}^p.
	\end{align*}
\end{corollary}

%
%
%

\end{document}